\documentclass
[a4paper,
11pt,
twoside
]{article}

\usepackage[T1]{fontenc}
\usepackage{ae}

\usepackage[left=1in, right=1in, top=1.5in, bottom=1in, headheight=15pt]{geometry}
\usepackage{authblk} 

\usepackage{amsmath}
\usepackage{amssymb}
\usepackage{amsthm}

 \usepackage{charter,courier} 
 \usepackage{chemarrow}
 \usepackage{caption} 

\usepackage[figuresright]{rotating} 

\usepackage{fancyhdr}
\usepackage{empheq}

\usepackage{latexsym}
\usepackage{lscape}

\usepackage{mleftright} 
\usepackage{makeidx}
\makeindex
\usepackage[nottoc]{tocbibind}

\usepackage[intoc]{nomencl}
\makenomenclature
\usepackage{stmaryrd}

\usepackage{textcomp}
\usepackage{tikz}

\usepackage{wasysym}

\usepackage{xcolor}
\usepackage{color}
\usepackage{colortbl}

\usepackage{yhmath}
\usepackage{graphicx}

\usepackage[all]{xy}

\definecolor{colComments}{rgb}{1,0,0}

\theoremstyle{plain}  

\swapnumbers 
\numberwithin{equation}{section} 
\newtheorem{Definition}[equation]{Definition} 
\newtheorem{Definition/Lemma}[equation]{Definition/Lemma}

\newtheorem{Lemma}[equation]{Lemma}
\newtheorem{Theorem}[equation]{Theorem}
\newtheorem{Proposition}[equation]{Proposition}
\newtheorem{Corollary}[equation]{Corollary}
\newtheorem{Remark}[equation]{Remark}

\newtheorem{Notation}[equation]{Notation}
\newtheorem{Notation/Lemma}[equation]{Notation/Lemma}

\newtheorem{Comparison}[equation]{Comparison}

\font\triangles=beta
\newcommand{\Squar}{\mbox{\triangles 3}}

\newcommand{\UR}{\mbox{\triangles 1}}

\newcommand{\UP}{\mbox{\triangles 2}}

\newcommand{\C}{\mbox{\triangles 9}}

\font\trianglesb=betab
\newcommand{\Squarb}{\mbox{\trianglesb 3}}
\newcommand{\URb}{\mbox{\trianglesb 1}}

\newcommand{\UPb}{\mbox{\trianglesb 2}}

\title{Supercharacter theories for Sylow $p$-subgroups
        \\ of the Ree groups}
\author{Yujiao Sun}
\affil{\footnotesize School of Mathematics and Statistics,
         Beijing Institute of Technology,\\
      Beijing 100081, P. R. China}
\affil{\footnotesize E-mail: yujiaosun@bit.edu.cn}

\date{}

\begin{document}
\maketitle

\begin{abstract}
We determine a supercharacter theory for Sylow $p$-subgroups ${^2{G}_2^{syl}(3^{2m+1})}$
of the Ree groups $^2{G}_2(3^{2m+1})$,
calculate the conjugacy classes of ${^2{G}_2^{syl}(3^{2m+1})}$,
and establish the character table of ${^2{G}_2^{syl}(3)}$.
\end{abstract}




\section{Introduction}
Let $p$ be a fixed odd prime,
$q$ a fixed power of $p$,
$\mathbb{F}_q$ the finite field with $q$ elements,
$\mathbb{N}^*$ the set of positive integers,
and $A_n(q)$ ($n\in \mathbb{N}^*$) the group of upper unitriangular $n\times n$-matrices
over $\mathbb{F}_q$.
Thus $A_n(q)$ is a Sylow $p$-subgroup \cite{Carter1}
of the Chevalley group of Lie type $A_{n-1}$ ($n \geq 2$) over $\mathbb{F}_q$.
It is known that determining the conjugacy classes of $A_n(q)$ for all $n$ and $q$ is a wild problem.
Higman's conjecture \cite{hig} states that
the number of conjugacy classes of $A_n(q)$
is an integer polynomial in $q$ depending on $n$.
Lehrer \cite{leh}
and later Isaacs \cite{isa2}
refined
Higman's conjecture.
However, Higman's conjecture is still open,
see e.g.
\cite{Evs2011, ps, vla}.

A generalization of the character theory, {\it supercharacter theory},
was introduced in \cite{di}.
The supercharacter theory replaces
irreducible characters,
conjugacy classes
and character table
by supercharacters,
superclasses
and
supercharacter table, respectively.
Andr\'{e} \cite{and1} using the Kirillov orbit method \cite{Kirillov},
and Yan \cite{yan2} using an algebraic and combinatorial method
determined the Andr\'{e}-Yan supercharacter theory for $A_n(q)$.
Andr\'{e} and Neto \cite{an1,an2, an3} studied the supercharacter theories
for the Sylow $p$-subgroups of untwisted types $B_n$, $C_n$ and $D_n$.
These supercharacters arise as restrictions of supercharacters of overlying
full upper unitriangular groups $A_N(q)$ to the Sylow $p$-subgroups,
and the superclasses arise as intersections of superclasses of $A_N(q)$ with these groups.
Marberg \cite{Mar2012} gave explicit definitions of
Andr\'{e}-Neto supercharacter theories of type $B_n$ and $D_n$.
The construction of \cite{an1, an3} has been extended to
Sylow $p$-subgroups of
finite classical groups of untwisted Lie type in a uniform way in \cite{AFM2015}.
Andrews \cite{Andrews2015, Andrews2016}
reproved the construction of \cite{an1, an2, an3}
and extended it once more to Sylow $p$-subgroups of twisted type $^2A_n$.

Jedlitschky
introduced the {\it monomial linearisation} method for a finite group
in his doctoral thesis \cite{Markus1}.
As a result, he
decomposed the {Andr\'{e}-Neto supercharacters} for Sylow $p$-subgroups
of Lie type $D_n$ into much smaller characters.
The smaller characters are pairwise orthogonal,
and each irreducible character is a constituent of exactly one of the smaller characters.
Thus these characters look like finer supercharacters
for the Sylow $p$-subgroups of type $D$.
So far there are no corresponding finer superclasses for the Sylow $p$-subgroups
of type $D$.
Recently, a monomial linearisation of the finite classical groups of untwisted type
is exhibited in \cite{Guo2017Orbit},
and some modules affording Andr\'{e}-Neto supercharacters are decomposed into a direct sum of submodules
in \cite{Guo2018On}.
One may ask, if there is a general supercharacter theory for
Sylow $p$-subgroups of all finite groups of Lie type
based on the monomial linearisation method.

The construction uses monomial linearisation to obtain supercharacters
and then supplements it
to establish superclasses as well in order to exhibit a full supercharacter theory.
The method seems to work for all Lie types,
indeed
the author determined
a full supercharacter theory
for the Sylow $p$-subgroup ${^3D}^{syl}_4(q^3)$ of twisted type ${^3D}_4$ in \cite{sun3D4super}
and for the Sylow $p$-subgroup $G_2^{syl}(q)$ of untwisted type $G_2$ in \cite{sunG2}.
In this paper,
we construct a full
supercharacter theory
for the Sylow $p$-subgroup ${^2{G}_2^{syl}(3^{2m+1})}$
of twisted type $^2G_2$.

In this paper,
we introduce Jedlitschky's construction of monomial modules in Section \ref{sec:The construction of monomial modules}.
For the matrix Sylow $p$-subgroup $U:={^2{G}_2^{syl}(3^{2m+1})}$ of the Ree group ${^2{G}_2(3^{2m+1})}$ (see Section \ref{sec:2G2-1}),
the explicit construction of a monomial $A_8(q)$-module $\mathbb{C}U$ is determined
in Section \ref{sec: monomial 2G2-module}.
In Section \ref{sec:U-orbit modules-2G2},
we classify the ${^2{G}_2^{syl}(3^{2m+1})}$-orbit modules
which lead to the supercharacters in Section \ref{sec: supercharacter theories-2G2}.
After that,
we calculate all of the conjugacy classes of ${^2{G}_2^{syl}(3^{2m+1})}$
in Section \ref{conjugacy classes-2G2}
which satisfy Higman's conjecture.
In Section \ref{conjugacy classes-2G2},
at the same time, we get a partition
which is proved to be a set of superclasses
in Section \ref{sec: supercharacter theories-2G2}.
In the last section,
we further determine the character table for the special case of ${^2{G}_2^{syl}(3)}$ (i.e. $m=0$).

Supercharacter theories have shown to be relevant to a number of areas of mathematics.
For example, Hendrickson \cite{Hend} obtained the connection between supercharacter theories and central Schur rings.
Brumbaugh et al. \cite{BBFGGKM2014} determined certain exponential sums of interest in number theory (e.g., Gauss, Ramanujan, Heilbronn, and Kloosterman sums) as supercharacters of abelian groups.

Here we fix some notations:
Let $K$ be a field,
$K^*$ the multiplicative group $K\backslash\{0\}$ of $K$,
$K^+$ the additive group of $K$,
$\mathbb{C}$ the complex field,
and
$\mathbb{N}$ the set of all non-negative integers.
Let $\mathrm{Mat}_{8 \times 8}(K)$
be the set of all $8\times 8$ matrices over $K$.
If $m\in \mathrm{Mat}_{8\times 8}(K)$, then set $m:=(m_{i,j})$,
where $m_{ij}:=m_{i,j}\in K$ denotes the $(i,j)$-entry of $m$.
We set $e_{ij}:=e_{i,j}\in \mathrm{Mat}_{8\times 8}(K)$ the matrix unit
with $1$ in the $(i,j)$-position and $0$ elsewhere.
Let $A^{\top}$ denote the {transpose} of $A\in \mathrm{Mat}_{8\times 8}(K)$.
Let $I_8$ be the $8 \times 8$ identity matrix $I_{8\times 8}$,
and $1$ be the identity element of a finite group.

\section{The construction of monomial modules}
\label{sec:The construction of monomial modules}
In this section, we recall the construction of the monomial modules,
and mainly refer to \cite{Markus1}.

Let $G$ be a
finite multiplicative group,
$\mathrm{Irr}(G)$ the set of all complex irreducible
characters of $G$,
$V$ a finite abelian additive group
and
$K$ a field.
If $V$ is a $K$-vector space, then let it be finite dimensional.
If $X$ is a set, then $KX$ denotes  the $K$-vector space with the $K$-basis $X$.
Let $M$ be a right $KG$-module
and $-*-$ be the module operation:
$-*- \colon M\times G  \to  M:(m,g)\mapsto m*g$.
Then the right $KG$-module $M$ is also denoted by
$(M,*)_{KG}$, or by $M_{KG}$ for short.

Let $K^G:=\{\tau \colon G \to  K \mid  \tau \text{ is a map}\}$.
Define addition and scalar multiplication on $K^G$ as follows:
For all $\tau, \sigma \in K^G$ and $\lambda\in K$,
we set
$(\tau+\sigma)(g)= \tau(g)+\sigma(g)$  and
$(\lambda\tau)(g)= \lambda(\tau(g))$
for all $g\in G$,
then $K^G$ is a $K$-vector space.
For $g\in G$, set
$\tau_g \colon G \to  K:h\mapsto
\left\{
\begin{array}{ll}
1,& g=h\\
0,& g\neq h
\end{array}
\right.
=\delta_{g,h}$,
where $\delta_{g,h}$ is the Kronecker delta.
We have that $\{\tau_g \mid  g\in G\}$ is a $K$-basis of $K^G$.
In particular, $\tau=\sum_{g\in G}\tau(g)\tau_g$ for all $\tau \in K^G$.
The map
$\varPhi \colon K^G  \to  KG$
induced by $\tau_g \mapsto g$
is a $K$-isomorphism.
In particular, $\varPhi(\tau)=\sum_{g\in G}\tau(g)g$ for all $\tau \in K^G$.
Let $KG$ be the group algebra with the multiplication
\begin{align*}
(\sum_{g\in G}{\alpha_g g})(\sum_{h\in G}{\beta_h h})
=\sum_{x\in G}\sum_{g\in G}{\alpha_g \beta_{g^{-1}x} x}.
\end{align*}
For $\tau, \sigma \in K^G$, the multiplication $\tau\sigma$ is defined by
\begin{align*}
\tau \sigma \colon  G \to  K
:y\mapsto \sum_{g\in G}{\tau(g)\sigma(g^{-1}y)},
\end{align*}
then $K^G$ is an associative $K$-algebra,
and $\varPhi \colon K^G \to KG: \tau \mapsto \sum_{g\in G}\tau(g)g$ is an algebra isomorphism.
In particular, $\tau_g\tau_h=\tau_{gh}$ for all $g,h\in G$.

For a finite abelian group $V$,
let $\hat{V}:=\mathrm{Hom}(V,\mathbb{C}^{*})$.
Then
$\mathrm{Irr}(V)=\hat{V}\subseteq \mathbb{C}^V$
is a linearly independent subset of the $\mathbb{C}$-vector space $\mathbb{C}^V$.
We have
$
 \dim_{\mathbb{C}} \mathbb{C}\hat{V}
 =|\hat{V}|
{=}|V|
 =\dim_{\mathbb{C}} \mathbb{C}{V}
 {=}\dim_{\mathbb{C}}\mathbb{C}^V$,
so $\mathbb{C}\hat{V}=\mathbb{C}^V$ (as $\mathbb{C}$-vector spaces).

\begin{Lemma}[see {\cite[\S 2.1]{Markus1}} and {\cite[2.6]{GuoD2018}}]\label{f*}
Let $f \colon G \to   V$ be a map.
Then
$f^* \colon \mathbb{C}^V \to  \mathbb{C}^G:\phi\mapsto f^*(\phi)=\phi f$
defines a $\mathbb{C}$-linear map,
and $f$ is surjective (bijective, injective)
if and only if $f^*$ is injective (bijective, surjective).
If $f$ is surjective,
then $\{\hat{\chi} f \mid \hat{\chi} \in \hat{V}\}$
 is a $\mathbb{C}$-basis of $\mathrm{im}{\, f^*}= f^*(\mathbb{C}^V)$.
\end{Lemma}

\begin{Corollary}\label{f|U*}
 Let $f \colon G \to   V$ be a surjective map,
 and  $U\leqslant G$ such that $f|_U$ is bijective.
 Then
$ f|_U^* \colon \mathbb{C}^V \to  \mathbb{C}^U:
  \phi\mapsto f|_U^*(\phi)=\phi f|_U=f^*(\phi)|_U$
defines a $\mathbb{C}$-isomorphism.
In particular,
 $\{\hat{\chi} f|_U \mid \hat{\chi} \in \hat{V}\}$
 is a $\mathbb{C}$-basis of $\mathbb{C}^U$.
\end{Corollary}

\begin{Definition}[1-cocycle]\label{1-cocycle}
Let $V$ be an abelian group. Suppose $G$ acts on $V$,
$(A,g)\mapsto A\circ g{\ }(A\in V, g\in G)$,
as automorphisms.
Then a map $f \colon G \to  V$
is called a (right) \textbf{1-cocycle}
\index{1-cocycle}
of $G$ in $V$
if it satisfies
\begin{align}
 f(xg)=f(x)\circ g+f(g) \qquad \text{for all $x,g\in G$}.
\end{align}
\end{Definition}

In the rest of this section,
suppose that  $f \colon G \to  V$ is a surjective 1-cocycle
and $U$ is a subgroup of $G$ such that $f|_U$ is bijective
(i.e. $f|_U$ is a bijective 1-cocyle of $U$ in $V$).
Then $\mathbb{C}^V$, $\mathbb{C}V$,
$\mathbb{C}^U$, $\mathbb{C}U$
and $\mathrm{im}{\,f^*}$
are pairwise
$\mathbb{C}$-isomorphic:
\begin{align*}
  \xymatrix{
\mathbb{C}\{ \hat{\chi}f|_U \mid \hat{\chi} \in \hat{V}\}=
\mathrm{im}{\,f|_U^*}=
&\mathbb{C}^U  \ar[r]^{\varPhi}
& \mathbb{C}U \ar[d]^{f_{\mathbb{C}U}} \\
\mathbb{C}\{ \hat{\chi} \mid \hat{\chi} \in \hat{V}\}=
\mathbb{C}\hat{V}=
& \mathbb{C}^V  \ar[u]^{f|_U^*} \ar[d]_{f^*}  \ar[r]^{\varPsi}
& \mathbb{C}V \\
\mathbb{C}\{ \hat{\chi}f \mid \hat{\chi} \in \hat{V}\}=
& \mathrm{im}{\,f^*}
&  \\
 }
\end{align*}
where
$\varPsi\colon \hat{\chi}\mapsto \sum_{B\in V}{\hat{\chi}(B)B}$,
$\varPhi\colon \hat{\chi}f|_U\mapsto  \sum_{u\in U}{\hat{\chi}f(u)u}$,
$f^*\colon  \hat{\chi}\mapsto \hat{\chi}f$,
$f|_U^*\colon  \hat{\chi}\mapsto \hat{\chi}f|_U$,
and $f_{\mathbb{C}U}:u\mapsto f(u)$ is the extension of $f|_U$ to $\mathbb{C}U$ by linearity.
Let $\chi:=\hat{\chi}f$,
\begin{align*}
[\hat{\chi}]:= \sum_{B\in V}{{\hat{\chi}(B)}B}
\quad
\text{ and }
\quad
[\hat{\chi}f|_U]:=[{\chi}|_U]:=\sum_{u \in U}{{{\chi}(u)}u}.
\end{align*}
Then
$
\mathbb{C}V=\mathbb{C}\{ [\hat{\chi}]\mid \hat{\chi} \in \hat{V}\}$
{ and }
$\mathbb{C}U=\mathbb{C}\{ [\hat{\chi}f|_U] \mid \hat{\chi} \in \hat{V}\}$.

Let $V,W$ be $K$-vector spaces (or abelian groups)
and
$\varphi \colon V \to  W$
a $K$-isomorphism (or group isomorphism).
Suppose that $V$ is a $KG$-module $(V,.)_{KG}$
   and that the elements of $G$ act on $V$ as $K$-automorphisms (or as group automorphisms).
Define a new operation by
$ -*- \colon W\times G \to  W:(w,g)\mapsto w*g:=\varphi(\varphi^{-1}(w).g)$.
We extend the operation by linearity
\begin{align*}
w*\Big(\sum_{g\in G}{\alpha_g g}\Big)=\sum_{g\in G}{\alpha_g (w*g)}
\quad \text{ for all }w \in W, {\ }\sum_{g\in G}{\alpha_g g}\in KG,
\end{align*}
then
$W$ is also a $KG$-module $(W,*)_{KG}$,
the elements of $G$ act on $W$ as $K$-automorphisms (or as group automorphisms)
and $\varphi$ is a $KG$-module isomorphism.

\begin{Definition/Lemma}\label{Gstar} 
Let $K$ be an arbitrary field and $G$ be a finite group.
Define an operation by
$-*- \colon K^G\times G \to  K^G:(\varphi,g)\mapsto \varphi*g$,
where $(\varphi*g)(x)=\varphi(xg^{-1})$ for all $x\in G$.
Then $K^G$ becomes a right $KG$-module $(K^G,*)_{KG}\cong KG_{KG}$,
where $KG_{KG}$ is the regular right module.
\end{Definition/Lemma}

\begin{Definition/Lemma}\label{hatVdot}
 Let $V$ be an abelian group on which $G$ acts from the right
 as automorphisms.
 Then the group action of $G$ on $V$ induces
 a group action $-.-$
 of $G$ on $\hat{V}$ given by
 \begin{align*}
  -.- \colon \hat{V}\times G  \to  \hat{V}:(\hat{\chi}, g)\mapsto \hat{\chi}.g,
 \end{align*}
where $(\hat{\chi}.g)(A)=\hat{\chi}(A\circ g^{-1})$ for all $A\in V$.
\end{Definition/Lemma}

\begin{Lemma}
\label{chi*g}
Let $f \colon G \to  V$ be a surjective 1-cocycle.
Then for $\hat{\chi} \in \hat{V}$ and $g\in G$,
\begin{align*}
 (\hat{\chi}f) *g= (\hat{\chi}f)(g^{-1})\cdot (\hat{\chi}.g)f
                 = \chi(g^{-1})\cdot (\hat{\chi}.g)f
                 \in \mathrm{im}{\,f^*},
\end{align*}
where $-\cdot-$ is the scalar multiplication.
By extending the operation $-*-$ linearly,
$\mathrm{im}{\,f^*}$ becomes a monomial module
$(\mathrm{im}{\,f^*}, *)_{\mathbb{C}G}$.

\end{Lemma}
\begin{proof}
 Let $\hat\chi\in \hat{V}$ and $g\in G$.
 Then for $x\in G$,
$((\hat{\chi}f)*g)(x)
     = (\hat{\chi}f)(xg^{-1})
 \stackrel{\ref{1-cocycle}}{=} \hat{\chi}(f(x)\circ g^{-1}+f(g^{-1}))
     =(\hat{\chi}f)(g^{-1}) \cdot \hat{\chi}(f(x)\circ g^{-1})
 \stackrel{\ref{hatVdot}}{=} (\hat{\chi}f)(g^{-1})\cdot  (\hat{\chi}.g)(f(x))
     = (\hat{\chi}f)(g^{-1})\cdot  (\big(\hat{\chi}.g\big)f)(x)
     = \chi(g^{-1})\cdot (\hat{\chi}.g)f$.
 Thus $(\hat{\chi}f) *g= (\hat{\chi}f)(g^{-1})\cdot (\hat{\chi}.g)f$.
\end{proof}

\begin{Theorem}
[Monomial $\mathbb{C}G$-modules, {[Jedlitschky, \cite[2.1.11]{Markus1}]} ]
\label{Markus, monomial theorem}
Let $f \colon G \to  V$ be a surjective 1-cocycle,
and $U\leqslant G$ such that $f|_U$ is bijective.
Then the $\mathbb{C}$-vector spaces
$\mathbb{C}^V$, $\mathbb{C}^U$,
$\mathbb{C}V$ and $\mathbb{C}U$
can be made into
monomial $\mathbb{C}G$-modules
by extending the following operations linearly:
for all $\hat{\chi}\in \hat{V}$ and $g\in G$, we have that
\begin{alignat*}{2}
\hat{\chi}*g:=&  \hat{\chi}\left(f(g^{-1})\right) \cdot \hat{\chi}.g,
&\qquad
(\hat{\chi}f|_U)*g:=&  \hat{\chi}\left(f(g^{-1})\right)\cdot (\hat{\chi}.g)f|_U, \\
[\hat{\chi}]*g:=&
          \hat{\chi}\left(f(g^{-1})\right)\cdot [\hat{\chi}.g],
& \qquad
[\hat{\chi}f|_U]*g:=&
          \hat{\chi}\left(f(g^{-1})\right)\cdot [(\hat{\chi}.g)f|_U],
\end{alignat*}
and that
$(\mathbb{C}^V,*)_{\mathbb{C}G}$,
$(\mathbb{C}V,*)_{\mathbb{C}G}$,
$(\mathbb{C}^U,*)_{\mathbb{C}G}$
and
$(\mathbb{C}U,*)_{\mathbb{C}G}$
are isomorphic to
$(\mathrm{im}{\,f^*},*)_{\mathbb{C}G}$.
We say these ${\mathbb{C}G}$-modules \textbf{arise from the 1-cocycle} $f$.
\end{Theorem}

\begin{Corollary}[Monomial $\mathbb{C}U$-modules]
The vector spaces $\mathbb{C}U$, $\mathbb{C}V$,
 $\mathbb{C}^U$,  $\mathbb{C}^V$,
 and $\mathrm{im}{\,f^*}$
can be made into
monomial isomorphic $\mathbb{C}U$-modules
by extending the restriction of the operations $-*-$ linearly.
In particular,
the operation $-*-$ of $U$ on $\mathbb{C}U$
is the usual right operation of $U$ on $\mathbb{C}U$, i.e.
for all $\hat{\chi}\in \hat{V}$ and $x\in U$, we obtain
\begin{align*}
 (\sum_{u\in U}{\hat{\chi}f(u)u})*x
 =\hat{\chi}\left(f(x^{-1})\right)\cdot (\sum_{u\in U}{(\hat{\chi}.x)f(u)u})
 =\sum_{u\in U}{\hat{\chi}f(u)ux},
\end{align*}
so $(\mathbb{C}U,*)_{\mathbb{C}U}=\mathbb{C}U_{\mathbb{C}U}$.
\end{Corollary}

\begin{Lemma}
Let $H:=\{g \in G\mid f(g)=0\}$. Then
$H \cap U =\{ 1 \}$ and $G=HU$.
\end{Lemma}
\begin{proof}
Since $H\leqslant G$ and $U \leqslant G$, we have $G \supseteq  HU$.

Let $X$ be a complete set of right coset representatives of $H$ in $G$.
If $g\in G$, then there exist $h\in H$  and $x\in X$ such that $g=hx$.
We have $f(g)=f(x)\in V$.
Then there exists $u\in U$ such that $f(u)=f(x)$ since $f|_U$ is bijective.
We know $f(xu^{-1})=f(x)\circ u^{-1}+f(u^{-1})
                   =f(u)\circ u^{-1}+f(u^{-1})
                   =f(1)=0$,
so $x=h_x u$ for some $h_x\in H$. Thus $g=hh_xu\in HU$, i.e. $G \subseteq HU$.
Therefore $G=HU$.

If $g\in H \cap U$, then $f(g)=0=f|_U(g)$.
So $g=1_G$ since $f|_U$ is bijective.
\end{proof}

\begin{Proposition}[see {\cite[2.8]{DG2016}}]
Let
\begin{align*}
 e:=& \sum_{h\in H}{h},\quad
 \tau_{e}:=  \sum_{h\in H}{\tau_h}\quad
 \text{ and}
 \quad
 [{\chi}]:=[\hat{\chi}f]:=\sum_{g\in G}{\chi(g)g} \quad \text{ for all } \hat{\chi}\in \hat{V},
 \end{align*}
and $\mathbb{C}_{H}=\mathbb{C}\{e\}$ be a trivial $H$-module.
Then
\begin{align*}
 \mathrm{Ind}_H^G \mathbb{C}_H
 =& e\mathbb{C}G
 =\mathbb{C}\{eu \mid u\in U\}
 =e\mathbb{C}U
 =\mathbb{C}\{e[\chi|_U] \mid \hat{\chi}\in \hat{V}\} \\
 =& \mathbb{C}\{[\chi] \mid \hat{\chi}\in \hat{V}\}
 \quad (\text{as $\mathbb{C}$-vector space}),\\
 \mathrm{im}{\,f^*}=& \mathbb{C}\{\chi \mid \hat{\chi}\in \hat{V}\}
                      = \mathbb{C}\{\tau_e*[\chi|_U] \mid \hat{\chi}\in \hat{V}\}
                      =\tau_e*\mathbb{C}U  \\
                   = & \mathbb{C}\{\tau_e*u \mid u\in U\}
                   =\tau_e*\mathbb{C}G
                   =\tau_e\mathbb{C}^G
 \quad (\text{as $\mathbb{C}$-vector space}).
\end{align*}
In particular,
$(\mathrm{im}{\,f^*},*)_{\mathbb{C}G}
  \cong \mathrm{Ind}_H^G \mathbb{C}_H$
  and
$\chi=\tau_e*[\chi|_U]$.
\end{Proposition}
\begin{proof}
It is enough to prove that $\chi=\tau_e*[\chi|_U]$. We have
\begin{align*}
 \chi=&\sum_{g\in G}{\chi(g)\tau_g}
     =\sum_{h\in H}\sum_{u\in U}{\chi(hu)\tau_{hu}}
     =\sum_{h\in H}\sum_{u\in U}{\chi(u)\tau_{h}*u}\\
     =&(\sum_{h\in H}\tau_h)*(\sum_{u\in U}{\chi(u)u})
     =\tau_e*[\chi|_U].
\end{align*}
\end{proof}

Now we give a summary of the isomorphic
monomial $\mathbb{C}G$-modules
and the $\mathbb{C}$-bases:
\begin{align*}
 \xymatrix{
 {\substack{\mathbb{C}\{ \chi|_U \mid \hat{\chi}\in \hat{V}\}=
             }}
&  (\mathbb{C}^U, *)_{\mathbb{C}G} \ar[r]^{\varPhi}
&  (\mathbb{C}U, *)_{\mathbb{C}G} \ar[d]^{f_{\mathbb{C}U}}
&  {\substack{=\mathbb{C}\{ [\chi|_U] \mid \hat{\chi}\in \hat{V}\}
     }}   \\
 {\substack{\mathbb{C}\hat{V}=}}
&  (\mathbb{C}^{V}, *)_{\mathbb{C}G} \ar[u]^{f|_U^*} \ar[d]_{f^*}  \ar[r]^{\varPsi}
&  (\mathbb{C}{V}, *)_{\mathbb{C}G} \ar[d]
&  {\substack{=\mathbb{C}\{[\hat{\chi}] \mid \hat{\chi}\in \hat{V}\}
     }} \\
{\substack{ \mathbb{C}\{ \tau_e*[\chi|_U] \mid \hat{\chi}\in \hat{V}\}=
           \\ \mathbb{C}\{ \chi \mid \hat{\chi}\in \hat{V}\}=}}
&  (\mathrm{im}{\,f^*},*)_{\mathbb{C}G} \ar[r]^{\varUpsilon_{\mathrm{im}{\,f^*}}}
&  \mathrm{Ind}_H^G{\mathbb{C}_H}
&  {\substack{=\mathbb{C}\{ e[\chi|_U] \mid \hat{\chi}\in \hat{V}\}\\
              =\mathbb{C}\{ [\chi] \mid \hat{\chi}\in \hat{V}\}\\}} \\
}
\end{align*}
where  $\varUpsilon\colon \mathbb{C}^G\to \mathbb{C}G
 :\tau \mapsto \sum_{g\in G}\tau(g)g$.

An example of the theory can be found in Section \ref{sec: monomial 2G2-module},
and $[\chi|_U]=[A]=[\chi_{-A}|_U]$ for all $A\in V$ in Theorem \ref{fund thm U-2G2}.



\section{Sylow $p$-subgroup ${^2{G}_2^{syl}(3^{2m+1})}$}
\label{sec:2G2-1}
In this section,
we construct a Sylow $3$-subgroup $^2{G}_2^{syl}(3^{2m+1})$
of the Ree group $^2{G}_2(3^{2m+1})$ (see \ref{sylow p subg-2G2}).

We recall the construction of the matrix Sylow $p$-subgroup ${G}_2^{syl}(q)$ of type $G_2$ (see \cite[\S 2]{sunG2}).

We firstly construct a matrix Lie algebra of type $G_2$
and a corresponding Chevalley basis.
Define the elements of $\mathrm{Mat}_{8\times 8}(\mathbb{C})$ as follows:
$ h_1:= e_{11}-e_{88}, \
 h_2:= e_{22}-e_{77}, \
 h_3:= e_{33}-e_{66}, \
 h_4:= e_{44}-e_{55}$.
A subspace of $\mathrm{Mat}_{8\times 8}(\mathbb{C})$ is
$\tilde{\mathcal{H}}:= \mathbb{C}
 \text{-span} \{h_1-h_2+2h_3,{\, }h_2-h_3\}
 =\{\sum_{i=1}^3\lambda_ih_i  \mid
\lambda_1-\lambda_2-\lambda_3=0,{\ }\lambda_i\in \mathbb{C}\}$.
Let $\tilde{\mathcal{H}}^*$ be the dual space of $ \tilde{\mathcal{H}}$,
$\tilde{h}:=\sum_{i=1}^3\lambda_ih_i  \in  \tilde{\mathcal{H}}$,
linear maps $\alpha \colon \tilde{\mathcal{H}} \to  \mathbb{C}
:\tilde{h}\mapsto \frac{\lambda_1-\lambda_2+2\lambda_3}{3}$
and
$\beta \colon \tilde{\mathcal{H}} \to  \mathbb{C}
:\tilde{h}\mapsto \lambda_2-\lambda_3$.
We set
$\Phi_{G_2}=\pm\{ \alpha,{\ } \beta,{\ }\alpha+\beta,
                  {\ }2\alpha+\beta, {\ }3\alpha+\beta, {\ }3\alpha+2\beta\}$,
and
$\Phi^{+}_{G_2}=\{ \alpha,{\ } \beta,{\ }\alpha+\beta,
                  {\ }2\alpha+\beta, {\ }3\alpha+\beta, {\ }3\alpha+2\beta\}$.
Let $\mathcal{V}_{G_2}$
be a $\mathbb{R}$-vector subspace of $\tilde{\mathcal{H}}^*$
spanned by $\Phi_{G_2}$,
 and become a Euclidean space (see \cite[\S 5.1]{Carter2005}.
Then $\Delta_{G_2}=\{\alpha,\  \beta\}$ is a basis of $\mathcal{V}_{G_2}$.
Define the elements of $\mathrm{Mat}_{8\times 8}(\mathbb{C})$ as follows:
\begin{alignat*}{2}
 e_\alpha:=& (e_{1,2}-e_{7,8})+(e_{3,4}-e_{5,6})+(e_{3,5}-e_{4,6}),
 & \quad
 e_{\beta}:=& e_{2,3}-e_{6 ,7},\\
 e_{\alpha+\beta}:=& -(e_{1,3}-e_{6,8})+(e_{2,4}-e_{5,7})+(e_{2,5}-e_{4,7}),
 & \qquad
 e_{3\alpha+\beta}:=& -(e_{1,6}-e_{3,8}),\\
 e_{2\alpha+\beta}:=& -(e_{1,4}-e_{5,8})-(e_{2,6}-e_{3,7})-(e_{1,5}-e_{4,8}),
 & \qquad
 e_{3\alpha+2\beta}:=& -(e_{1,7}-e_{2,8}),
\end{alignat*}
and $e_{-r}:=e_r^\top$ and $h_r:=[e_r, e_{-r}]$ for all $r\in \Phi^+$.
Then
a Lie algebra of type $G_2$ is determined, denoted by $\mathcal{L}_{G_2}$,
which has a Chevalley basis
$\{h_\alpha,{\,} h_\beta \}\cup \{e_r  \mid  r\in \Phi\}$
(see \cite[2.1]{sunG2}).
Let $r:=x_1\alpha+x_2\beta \in \mathcal{V}_{G_2}$, $s:=y_1\alpha+y_2\beta \in \mathcal{V}_{G_2}$.
Then we write
$ r\prec s$,
if $\sum_{i=1}^2{x_i}<\sum_{i=1}^2{y_i}$,
or if $\sum_{i=1}^2{x_i}= \sum_{i=1}^2{y_i}$ and
the first non-zero coefficient $x_i-y_i$ is positive.
The total order on $\Phi^{+}_{G_2}$ is determined:
$
 0\prec
 \alpha \prec \beta \prec
 \alpha+\beta \prec
 2\alpha+\beta \prec
 3\alpha+\beta \prec
 3\alpha+2\beta $.
The Lie algebra $\mathcal{L}_{G_2}$ has the following structure constants:
$N_{\alpha, \beta }=-1$,
$N_{\alpha, \alpha+\beta}=-2$,
$N_{\alpha, 2\alpha+\beta}=3$
and
$N_{\beta, 3\alpha+\beta}=1$.
In particular,
$N_{\alpha,\alpha+\beta}=-2N_{\beta,3\alpha+\beta}$
and $N_{\alpha,2\alpha+\beta}=3N_{\beta,3\alpha+\beta}$.

Set a matrix group
$
{\bar{G}_2}(q):= \left<{\,} \exp(te_r) {\, }\middle|{\, }
                       r\in \Phi_{G_2},{\ } t\in \mathbb{F}_q {\,}\right>$,
and the Chevalley group of type $\mathcal{L}_{G_2}$ over the field $\mathbb{F}_q$ is
$
G_2(q):= \left<{\,} \exp(t{\,}\mathrm{ad}{\,e_r})
                {\, }\middle|{\, } r\in \Phi_{G_2},{\ } t\in \mathbb{F}_q {\,}\right>$.
For all $r\in \Phi_{G_2}$ and $t\in \mathbb{F}_q$,
set
$ y_r(t):=\exp(te_r)=I_8+te_r+\frac{1}{2}t^2e_r^2$.
Let
 $y_1(t):= y_{\alpha}(t)$,
 $y_2(t):= y_{\beta}(t)$,
 $y_3(t):= y_{\alpha+\beta}(t)$,
$y_4(t):=y_{2\alpha+\beta}(t)$,
$y_5(t):=y_{3\alpha+\beta}(t)$,
$y_6(t):=y_{3\alpha+2\beta}(t)$.
The positive root subgroups of $G_2(q)$ are
$Y_i:=\left\{y_i(t) {\, }\middle|{\, } t\in \mathbb{F}_q \right\}$
for all $i=1,2,\dots,6$.

Let $t_i \in \mathbb{F}_{q}$ for all $i=1,2,\dots,6$
and
$[y_i(t_i), y_j(t_j)]:=y_i(t_i)^{-1}y_j(t_j)^{-1}y_i(t_i)y_j(t_j)$.
Then the non-trivial commutators are determined.
\begin{align*}
[y_1(t_1), y_2(t_2)]=& y_3(-t_2t_1)\cdot y_4(-t_2t_1^{2})
                                   \cdot y_5(t_2t_1^{3})
                                   \cdot y_6(-2t_2^2t_1^{3}),\\
[y_1(t_1), y_3(t_3)]=& y_4(-2t_1t_3)
                       \cdot y_5(3t_1^{2}t_3)
                       \cdot y_6(3t_1t_3^2),\\
[y_1(t_1), y_4(t_4)]=& y_5(3t_1t_4),\qquad
[y_3(t_3), y_4(t_4)]= y_6(3t_3t_4),\qquad
[y_2(t_2), y_5(t_5)]= y_6(t_2t_5).
\end{align*}
In particular, if $\mathrm{Char}{\,\mathbb{F}_q}=3$, then
\begin{align*}
[y_1(t_1), y_2(t_2)]
                    =& y_3(-t_2t_1)\cdot y_4({-t_2t_1^{2}})
                                   \cdot y_5({t_2t_1^{3}})
                                   \cdot y_6({t_2^2t_1^{3}}),\\
[y_1(t_1), y_3(t_3)]=& y_4(-2t_1t_3)
                    = y_4(t_1t_3),\qquad
[y_2(t_2), y_5(t_5)]= y_6(t_2t_5).
\end{align*}

 Let
 $y(t_1, t_2, t_3, t_4, t_5, t_6):=
   y_2(t_2)y_1(t_1)y_3(t_3)y_4(t_4)y_5(t_5)y_6(t_6)$
for all $t_i\in \mathbb{F}_q$ $(i=1,2,\dots,6)$.
Then a matrix Sylow $p$-subgroup ${G}_2^{syl}(q)$ of $G_2(q)$ (see \cite[2.6]{sunG2})
is
\begin{align*}
  {G}_2^{syl}(q)
 :=\left\{y(t_1, t_2, t_3, t_4, t_5, t_6)
 \mid
t_1, t_2, t_3, t_4, t_5, t_6 \in {\mathbb{F}_{q}}\right\}.
\end{align*}

Note that
the signs of the structure constants and the Chevalley basis of the Lie algebra $\mathcal{L}_{G_2}$
in this paper
are different from those in \cite[\S 2]{sunG2}:
the structure constant
$N_{\alpha, \alpha+\beta}=-2$ in this paper
while
$N_{\alpha, \alpha+\beta}=2$ in \cite{sunG2};
$e_{2\alpha+\beta}$,
$e_{3\alpha+\beta}$
and
$e_{3\alpha+2\beta}$
are in the Chevalley basis
in this paper
while $-e_{2\alpha+\beta}$,
$-e_{3\alpha+\beta}$
and
$-e_{3\alpha+2\beta}$
are in the Chevalley basis in \cite{sunG2}.
However, the Sylow $p$-subgroup $G_2^{syl}(q)$ and
the root subgroups $Y_i$  $(i=1,2,\dots,6)$
of the Chevalley group $G_2(q)$
are as same as those in \cite{sunG2}.


With the above construction for type $G_2$,
we can determine a matrix Sylow $3$-subgroup $^2{G}_2^{syl}(q)$
for the twisted type $^2{G}_2$.

Let $p:=3$, $q:=3^{2m+1} {\ }(m\in \mathbb{N})$ and $\theta:=3^m$.
There is a field automorphism $F_\theta$ of $G_2(q)$
sending $y_r(t)$ to $y_r(t^\theta)=y_r(t^{3^m})$
for all $r\in \mathrm{\Phi}_{G_2}$.
Let $\rho \colon r\mapsto \bar{r}$ be a non-trivial symmetry of the Dynkin diagram of type $G_2$
(interchanging $\alpha$ and $\beta$).
For every $r\in \Phi_{G_2}$,
$\bar{r}$ is obtained by reflecting $r$ in the line bisecting $\alpha$ and $\beta$.
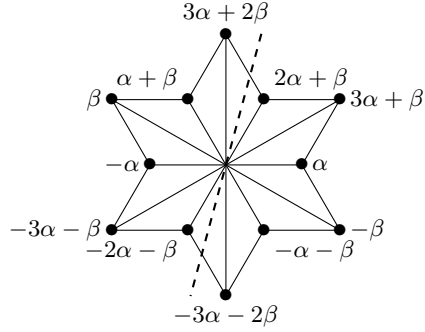
\begin{center}
\begin{tikzpicture}[scale=1.0]
 \node (0) at (0:0) {};
 \node (a)   at (0:1) {$\bullet$};
 \node (2ab) at (60:1) {$\bullet$};
 \node (ab)  at (120:1) {$\bullet$};

 \node (-a)   at (180:1) {$\bullet$};
 \node (-2ab) at (-120:1) {$\bullet$};
 \node (-ab)  at (-60:1) {$\bullet$};

 \node (3ab)  at (30:1.732) {$\bullet$};
 \node (3a2b) at (90:1.732) {$\bullet$};
 \node (b)    at (150:1.732) {$\bullet$};

 \node (-3ab)  at (-150:1.732) {$\bullet$};
 \node (-3a2b) at (-90:1.732) {$\bullet$};
 \node (-b)    at (-30:1.732) {$\bullet$};

 \draw (-a.center) node[left] {\footnotesize $-\alpha$}
                   to  (a.center) node[right] {\footnotesize $\alpha$};
 \draw (b.center)  node[left] {\footnotesize $\beta$}
                   to  (-b.center) node[right] {\footnotesize $-\beta$};
 \draw (ab.center) node[above left] {\footnotesize $\alpha+\beta$}
                   to  (-ab.center) node[below right] {\footnotesize $-\alpha-\beta$};
 \draw (3a2b.center) node[above] {\footnotesize $3\alpha+2\beta$}
                     to  (-3a2b.center) node[below] {\footnotesize $-3\alpha-2\beta$};
 \draw (2ab.center) node[above right] {\footnotesize $2\alpha+\beta$}
                    to  (-2ab.center) node[below left] {\footnotesize $-2\alpha-\beta$};
 \draw (3ab.center) node[right] {\footnotesize $3\alpha+\beta$}
                    to  (-3ab.center) node[left] {\footnotesize $-3\alpha-\beta$};
 \draw (b.center) to (-a.center);
 \draw (b.center) to (ab.center);
 \draw (-b.center) to (a.center);
 \draw (-b.center) to (-ab.center);

 \draw (3a2b.center) to (ab.center);
 \draw (3a2b.center) to (2ab.center);
 \draw (-3a2b.center) to (-ab.center);
 \draw (-3a2b.center) to (-2ab.center);

 \draw (3ab.center) to (2ab.center);
 \draw (3ab.center) to (a.center);
 \draw (-3ab.center) to (-2ab.center);
 \draw (-3ab.center) to (-a.center);
 \node (c)   at (75:1.8)   {};
 \node (-c)  at (-105:1.8) {};
 \draw [dashed,thick] (c.center) to (-c.center);
\end{tikzpicture}
\end{center}
Let $\epsilon_i=\pm 1$ ($i=1,2,3,4$) satisfy that
$N_{\alpha, \beta}= \epsilon_1$,
$N_{\alpha, \alpha+\beta}= 2\epsilon_2$,
$N_{\alpha, 2\alpha+\beta}= 3\epsilon_3$
and
$N_{\beta, 3\alpha+\beta}= \epsilon_4$.
Since the structure constants of $\mathcal{L}_{G_2}$
satisfy
$ -\epsilon_2=\epsilon_3= \epsilon_4=1$,
by \cite[12.4.1]{Carter1}
the map
\begin{align*}
 {y_r(t)}\mapsto
{\left\{
\begin{array}{ll}
y_{\bar{r}}(t), & \bar{r} \text{ is short}\\
y_{\bar{r}}(t^{3}), & \bar{r} \text{ is long}\\
\end{array}
\right.}
\qquad \text{for all }r\in \Phi_{\mathrm{G}_2},{\ }t\in \mathbb{F}_q
\end{align*}
can be extended to a graph automorphism $\tilde{\rho}$ of $G_2(q)$.
If $F:=\tilde{\rho} {F_\theta}={F_\theta} \tilde{\rho}$,
then
\begin{align*}
 F \colon  G_2(q) \to  G_2(q):
 {y_r(t)}\mapsto
{\left\{
\begin{array}{ll}
y_{\bar{r}}(t^\theta), & \bar{r} \text{ is short}\\
y_{\bar{r}}(t^{3\theta}), & \bar{r} \text{ is long}\\
\end{array}
\right.                         }
\qquad \text{for all }r\in \Phi_{\mathrm{G}_2},{\ }t\in \mathbb{F}_q.
\end{align*}
For a subgroup $X$ of $G_2(q)$,
we write $X^F:=\{x\in X \mid F(x)=x\}$.
By \cite[\S 13.4]{Carter1}, $G_2(q)^F={^2}G_2(q)$
and ${G^{syl}_2(q)}^F$ is a subgroup of ${^2}G_2(q)$.
By \cite[\S 14]{Carter1},
we have that
${G^{syl}_2(q)}^F$ is also a Sylow $p$-subgroup of ${^2}G_2(q)$
and
that $ |^2{G}_2(q)|= q^3(q-1)(q^3+1)
     = 3^{6m+3}(3^{2m+1}-1)(3^{6m+3}+1)$.

\begin{Proposition}[Sylow $p$-subgroup $^2{G}_2^{syl}(3^{2m+1})$]
\label{sylow p subg-2G2}
Let $p=3$, $q=3^{2m+1} {\ }(m\in \mathbb{N})$ and $\theta=3^m$.
Then a matrix Sylow $3$-subgroup $^2{G}_2^{syl}(q)$ of the Ree group $^2{G}_2(q)$ is
\begin{align*}
^2{G}_2^{syl}(q): =&
\left\{
y_2({t_1^{3\theta}})\cdot
y_1({t_1})
y_3({t_3})
y_4({t_4})
\cdot
y_5({t_3^{3\theta}}+{t_1^{3\theta+3}})
y_6({t_4^{3\theta}}+{t_1^{6\theta+3}})
{\, }\middle|{\, }
t_1,t_3,t_4 \in \mathbb{F}_q\right\},
\end{align*}
where
\begin{align*}
& y_2({t_1^{3\theta}})\cdot
y_1({t_1})
y_3({t_3})
y_4({t_4})
\cdot
y_5({t_3^{3\theta}}+{t_1^{3\theta+3}})
y_6({t_4^{3\theta}}+{t_1^{6\theta+3}})\\
=&
\left(
\newcommand{\mc}[3]{\multicolumn{#1}{#2}{#3}}
\begin{array}{cccccccc}\cline{2-4}
\mc{1}{c|}{1}
& \mc{1}{c|}{\begin{array}{l}
              {t_1}
             \end{array} }
& \mc{1}{c|}{\begin{array}{l}
              {-t_3}
             \end{array}}
& \mc{1}{c|}{\begin{array}{l}
 t_1t_3\\
 -t_4
\end{array}}
& \begin{array}{l}
 t_1t_3\\
 -t_4
\end{array}
& \begin{array}{l}
 -t_1t_4\\
-t_1^{3\theta+3}\\
-t_3^{3\theta}
\end{array}
& \rule{0pt}{33pt}
\begin{array}{l}
 -t_1t_3^2\\
-t_3t_4\\
-t_1^{6\theta+3}\\
-t_4^{3\theta}
 \end{array}
& \begin{array}{l}
2t_1{t_3}{t_4}
+t_1^{6\theta+4}\\
+t_1t_4^{3\theta}
-t_3t_1^{3\theta+3}\\
-t_3^{3\theta+1}
-{t_4^2}
\end{array}\\\cline{2-4}
×
& {1}
& {t_1^{3\theta}}
& \begin{array}{l}
t_1^{3\theta+1}\\
+{t_3}
\end{array}
& {\begin{array}{l}
t_1^{3\theta+1}\\
+{t_3}
\end{array}}
& \begin{array}{l}
-t_1^{3\theta+2}\\
-{t_4}
\end{array}
& \begin{array}{l}
-2t_1^{3\theta+1}{t_3}\\
+t_1^{3\theta}{t_4}\\
-{t_3^2}\\
  \end{array}
& \rule{0pt}{37pt}
\begin{array}{l}
-t_1^{3\theta+2}t_3\\
+2t_1^{3\theta+1}{t_4}\\
+t_1^{3\theta}t_3^{3\theta}
+2{t_3}{t_4}\\
+t_4^{3\theta}
+2t_1^{6\theta+3}
\end{array}\\
× & × & {1} & {t_1}
& {t_1}
& {-{t_1^2}}
& \begin{array}{l}
-2{t_1}{t_3}\\
+{t_4}
  \end{array}
& \rule{0pt}{26pt}
\begin{array}{l}
  -{t_1^2}t_3
+2{t_1}{t_4}\\
+t_3^{3\theta}
+t_1^{3\theta+3}\\
  \end{array}\\
× & × & × & 1 & 0 & -{t_1} & \rule{0pt}{15pt} -{t_3}
& \rule{0pt}{13pt}
-{t_1}t_3+{t_4} \\
× & × & × & × & 1 & -{t_1} &-t_3
& -{t_1}t_3
    +t_4\\
× & × & × & × & × & 1
& -t_1^{3\theta}
& t_1^{3\theta+1}
    +t_3
\\
× & × & × & × & × & × & 1
& -t_1\\
× & × & × & × & × & × & × & 1\\
\end{array}
\right).
\end{align*}
\end{Proposition}


\begin{proof}
We know $\mathrm{Char}{\,\mathbb{F}_q}=3$ and $t^{3\theta^2}=t$ for all $t\in \mathbb{F}_q$.
Let $t_i\in \mathbb{F}_q$ and $y(t_1,t_2,t_3,t_4,t_5,t_6)\in {G_2^{syl}(q)}^{F}$.
Then
\begin{align*}
& y(t_1,t_2,t_3,t_4,t_5,t_6)
=  F\big(y(t_1,t_2,t_3,t_4,t_5,t_6)\big)\\
=&  F\big(y_2(t_2)y_1(t_1)y_3(t_3)y_4(t_4)y_5(t_5)y_6(t_6)\big)
=  y_1(t_2^\theta)y_2(t_1^{3\theta})y_5(t_3^{3\theta})y_6(t_4^{3\theta})y_3(t_5^\theta)y_4(t_6^\theta)\\
{=}
 &  y_2(t_1^{3\theta})y_1(t_2^\theta)
    y_3(-t_1^{3\theta}t_2^{\theta})
    y_4(-t_1^{3\theta}t_2^{2\theta})
    y_5(t_1^{3\theta}t_2^{3\theta})
    y_6(t_1^{6\theta}t_2^{3\theta})
  \cdot y_3(t_5^\theta)y_4(t_6^\theta)y_5(t_3^{3\theta})y_6(t_4^{3\theta})\\
=& y\big(
    t_2^\theta,{\ }
    t_1^{3\theta},{\ }
    t_5^\theta-t_1^{3\theta}t_2^{\theta},{\ }
    t_6^\theta-t_1^{3\theta}t_2^{2\theta},{\ }
    t_3^{3\theta}+t_1^{3\theta}t_2^{3\theta},{\ }
    t_4^{3\theta}+t_1^{6\theta}t_2^{3\theta}
    \big).
\end{align*}
Thus,
\begin{alignat*}{3}
 t_1=& t_2^\theta,
& \quad
 t_3=& t_5^\theta-t_1^{3\theta}t_2^{\theta}=t_5^\theta-t_1^{3\theta+1},
& \quad
 t_5=& t_3^{3\theta}+t_1^{3\theta}t_2^{3\theta}=t_3^{3\theta}+t_1^{3\theta+3},\\
t_2=& t_1^{3\theta},
& \quad
t_4=& t_6^\theta-t_1^{3\theta}t_2^{2\theta}
                               =t_6^\theta-t_1^{3\theta+2},
& \quad
t_6=& t_4^{3\theta}+t_1^{6\theta}t_2^{3\theta}
                               =t_4^{3\theta}+t_1^{6\theta+3}.
\end{alignat*}
Hence
$ y(t_1,t_2,t_3,t_4,t_5,t_6)
= y\big(
    t_1,{\, }
    t_1^{3\theta},{\, }
    t_3,{\, }
    t_4,{\, }
    t_3^{3\theta}+t_1^{3\theta+3},{\,}
    t_4^{3\theta}+t_1^{6\theta+3}
    \big)$,
and $\big| {G_2^{syl}(q)}^{F} \big |=q^3$.

Therefore, $^2{G}_2^{syl}(q):={G_2^{syl}(q)}^{F}$ is a Sylow $p$-subgroup of $^2{G}_2(q)$.
We get the matrix form by calculation.
\end{proof}

\begin{Corollary}
 $^2{G}_2^{syl}(3^{2m+1}) \leqslant
  {G}_2^{syl}(3^{2m+1})$.
\end{Corollary}

\begin{Notation}
 For $i\in \{1,3,4\}$ and $t_i\in \mathbb{F}_q$,
 we set
\begin{align*}
 a(t_1):= y_2(t_1^{3\theta}) y_1(t_1)y_5(t_1^{3\theta+3})y_6(t_1^{6\theta+3}),\quad
 b(t_3):= y_3(t_3)y_5(t_3^{3\theta}),\quad
 c(t_4):=y_4(t_4)y_6(t_4^{3\theta}).
\end{align*}
\end{Notation}
\begin{Lemma}
Let
$Y(t_1,t_3,t_4):=a(t_1)b(t_3)c(t_4)$.
{Then}
\begin{align*}
& y_2(t_1^{3\theta})
\cdot
y_1(t_1)
y_3(t_3)
y_4(t_4)
\cdot
y_5(t_3^{3\theta}+t_1^{3\theta+3})
y_6(t_4^{3\theta}+t_1^{6\theta+3})
= a(t_1) b(t_3) c(t_4)
= Y(t_1,t_3,t_4).
\end{align*}
\end{Lemma}

By calculation, we get the following properties.
\begin{Lemma}
 Let $i\in \{1,3,4\}$ and $t_i, s_i\in \mathbb{F}_q$. Then
\begin{align*}
 Y(t_1,t_3,t_4)\cdot Y(s_1,s_3,s_4)\
=& Y(t_1+s_1,{\ }t_3+s_3-t_1s_1^{3\theta},{\ }t_4+s_4+t_1s_1^{3\theta+1}-t_1^2s_1^{3\theta}-t_3s_1),\\
 Y(t_1,t_3,t_4)^{-1}
=& Y(-t_1,{\ }-t_3-t_1^{3\theta +1},{\ } -t_4+t_1^{3\theta+2}-t_1t_3).
\end{align*}
In particular,
\begin{alignat*}{2}
 & a(t_1)\cdot a(s_1)= Y(t_1+s_1,{\ }-t_1s_1^{3\theta},{\ }t_1s_1^{3\theta+1}-t_1^2s_1^{3\theta}),
 & \quad
 & a(t_1)^{-1}= Y(-t_1,{\ } -t_1^{3\theta +1},{\ } t_1^{3\theta+2}),\\
 & b(t_3)\cdot b(s_3)= b(t_3+s_3),
 & \quad
 & c(t_4)\cdot c(s_4)= c(t_4+s_4).
\end{alignat*}
\end{Lemma}

\begin{Lemma}
If $i\in \{1,3,4\}$ and $t_i, s_i\in \mathbb{F}_q$,
then the commutators of $^2{G}_2^{syl}(q)$ are
\begin{align*}
 & [Y(t_1,t_3,t_4), Y(s_1,s_3,s_4)]\\
=& Y\big(0,{\ }t_1^{3\theta}s_1-t_1s_1^{3\theta},{\ }
       (t_1s_1^{3\theta+1}-t_1^{3\theta+1}s_1)+(t_1^{3\theta}s_1^2-t_1^2s_1^{3\theta})+(t_1s_3-t_3s_1)\big),\\
 & [Y(t_1,t_3,t_4)^{-1}, Y(s_1,s_3,s_4)^{-1}]
= Y\big(0,{\ }t_1^{3\theta}s_1-t_1s_1^{3\theta},{\ }
       (t_1^2s_1^{3\theta}-t_1^{3\theta}s_1^2)+(t_1s_3-t_3s_1)\big).
\end{align*}
{In particular, }
\begin{align*}
& [a(t_1), a(s_1)]= b(t_1^{3\theta}s_1-t_1s_1^{3\theta})
     \cdot c(t_1^{3\theta}s_1^2-t_1^2s_1^{3\theta}
           +t_1s_1^{3\theta+1}-t_1^{3\theta+1}s_1),\\
& [a(t_1)^{-1}, a(s_1)^{-1}]= b(t_1^{3\theta}s_1-t_1s_1^{3\theta})
     \cdot c(t_1^2s_1^{3\theta}-t_1^{3\theta}s_1^2),\\
& [a(t_1), b(s_3)]= c(t_1s_3),\qquad
 [a(t_1)^{-1}, b(s_3)^{-1}]= c(t_1s_3).
\end{align*}
\end{Lemma}

\begin{Proposition}\label{prop:conjugacy classes of a,b,c-2G2}
Let $t_i, s_i\in \mathbb{F}_q$ with $i\in \{1,3,4\}$.
Then the conjugate of $Y(t_1,t_3,t_4)$ is
 \begin{align*}
& Y(s_1,s_3,s_4)\cdot Y(t_1,t_3,t_4)\cdot Y(s_1,s_3,s_4)^{-1}\\
=& Y\big(t_1,{\ }t_3+t_1s_1^{3\theta}-t_1^{3\theta}s_1,{\ }
t_4+(t_1^2s_1^{3\theta}+t_1^{3\theta}s_1^2)+t_1^{3\theta+1}s_1+(t_3s_1-t_1s_3)\big).
\end{align*}
In particular,
\begin{align*}
& Y(s_1,s_3,s_4)\cdot a(t_1)\cdot Y(s_1,s_3,s_4)^{-1}
= Y\big(t_1,{\ }t_1s_1^{3\theta}-t_1^{3\theta}s_1,{\ }
(t_1^2s_1^{3\theta}+t_1^{3\theta}s_1^2)+t_1^{3\theta+1}s_1-t_1s_3\big),\\
& Y(s_1,s_3,s_4)\cdot b(t_3)\cdot Y(s_1,s_3,s_4)^{-1}= Y(0,{\ }t_3,{\ }t_3s_1),\\
& Y(s_1,s_3,s_4)\cdot c(t_4)\cdot Y(s_1,s_3,s_4)^{-1}= c(t_4).
\end{align*}
\end{Proposition}

 Define the following sets of matrix entry coordinates:
  $\Squar:= \{(i,j) \mid  1\leq i,j \leq 8 \}$,
  $\UR:= \{(i,j) \mid 1\leq i< j \leq 8 \}$
 and
  $\UP:= \{(i,j)\in \Squar  \mid  i < j < 9-i \}$.
Let
 $\tilde{J}:=\{(1,2), (1,3), (1,4), (1,5), (1,6), (1,7), (2,3)\} \subseteq \UP$,
and
 $J:=\{(1,2), (1,3), (1,4)\} \subseteq \tilde{J}$
\begin{Comparison}[Sylow $p$-subgroups]
\label{com:sylow-2G2}
For every element of $^2G_2^{syl}(q)$  in \ref{sylow p subg-2G2},
we have matrix entries $t_1$ and
 up to sign also $t_3$ with positions in $J$,
 but $t_4$
 appears in $J$ in
 polynomials
 involving $t_1$ and $t_3$.
This is similar to
that of ${{^3D}_4^{syl}}(q^3)$ (see \cite[\S 2]{sun3D4super})
and that of $G_2^{syl}(q)$ (see \cite[\S 2]{sunG2}).
\end{Comparison}



\section{Monomial ${{^2}G}^{syl}_2(3^{2m+1})$-module}
\label{sec: monomial 2G2-module}
Let
$p:=3$, $q:=3^{2m+1} {\ }(m\in \mathbb{N})$,
$G:=A_8(q)$
and
$U:={{^2}G}^{syl}_2(3^{2m+1})$.
In this section,
we explain the construction of a monomial $A_8(q)$-module $\mathbb{C}U$ (see \ref{fund thm U-2G2})
that is analogous to that of ${{^3D}_4^{syl}}(q^3)$ (see \cite{sun3D4super})
and that of $G_2^{syl}(q)$ (see \cite{sunG2}).
The construction is further applied to the cases
of type $E_6$, $E_7$, $F_4$ and the twisted type ${^2}E_6$
in the author's PhD thesis \cite[Chapter 5]{sunphd}.

Let $V_0:=\mathrm{Mat}_{8\times 8}(q)$.
For any subset $I\subseteq \Squar$,
let $V_I:=\bigoplus_{(i,j)\in I}{\mathbb{F}_{q}}e_{ij} \subseteq V_0$.
In particular, $V_{\Squarb}=V_0$.
Then $V_I$ is an $\mathbb{F}_{q}$-vector subspace.
We have $\dim_{\mathbb{F}_q}{V_J}=3$,
since
 $J=\{(1,2), (1,3), (1,4)\}$.
The map
$\kappa \colon  V_0\times V_0  \to  {\mathbb{F}_{q}}: (A,B)\mapsto \mathrm{tr} (A^\top B)$
is a non-degenerate symmetric $\mathbb{F}_{q}$-bilinear form on $V_0$
which is called the {\textbf{trace form}}.
Let $V:=V_J$,
and $V^{\bot}$ denote the orthogonal complement of $V$ in $V_0$
 with respect to the trace form $\kappa$, i.e.
 $ V^{\bot}:=\{B\in V_0 \mid\kappa(A,B)=0,{\ }\forall {\ } A \in V\}$.
 Then
$V^{\bot}=V_{\Squarb \backslash J}$
and
$V_0=V \oplus V^{\bot}$.
Note that $\kappa|_{V\times V} \colon V\times V \to  \mathbb{F}_q$ is a non-degenerate bilinear form.
Set $\pi:=\pi_J$, i.e.
\begin{align*}
 \pi \colon V_0=V\oplus V^\bot  \to  V:
  A\mapsto
\sum_{(i,j)\in J}{A_{i,j}e_{i,j}}
=A_{12}e_{12}+A_{13}e_{13}+A_{14}e_{14}.
\end{align*}
Then $\pi$ is a projection to the first component
$V$
and is an ${\mathbb{F}_q}$-linear map.
Suppose $A,B\in V_0$ such that $\mathrm{supp}(A)\cap\mathrm{supp}(B)\subseteq J$.
Then
$\kappa(A,B)=\kappa(\pi(A),B)=\kappa(A,\pi(B))
= \kappa(\pi(A),\pi(B))=\kappa|_{V\times V}(\pi(A),\pi(B))$.
If $A,B\in V$ and $g,h\in G$, then
$\pi_{\URb}(Ag^\top)\in V$
and
$\mathrm{supp}(Bh^\top)\cap \mathrm{supp}(Ag)\subseteq  J$.
\begin{Proposition}[Group action of $G$ on $V$]\label{circ action-2G2}
The map
\begin{align*}
-\circ- \colon V\times G \to  V: (A,g)\mapsto A\circ g:=\pi(Ag)
\end{align*}
is a group action,
and the elements of the group $G$ act as
$\mathbb{F}_q$-automorphisms.
\end{Proposition}
\begin{proof}
Let $A, B\in V$, $g, h\in G$.
Since $\pi$ is $\mathbb{F}_q$-linear,
it is enough to prove
$A\circ (gh) {=}(A\circ g) \circ h$.
We have that
\begin{align*}
 &  \kappa(B, A\circ(gh))
{=}\kappa(B, A(gh))
= \kappa(Bh^\top, Ag)
{=} \kappa(\pi(Bh^\top), Ag)\\
{=}&\kappa(\pi(Bh^\top), A\circ g)
{=} \kappa(Bh^\top, A\circ g)
=\kappa(B, (A\circ g)h)
{=}\kappa(B, (A\circ g)\circ h).
\end{align*}
Since $\kappa|_{V\times V}$ is a non-degenerate bilinear form,
we obtain that $A\circ (gh) {=}(A\circ g) \circ h$.
Therefore, the proof is completed.
\end{proof}

\begin{Corollary}\label{circ g-circ gT-2G2}
If $A,B\in V$ and $g\in G$,
then we have that
$\kappa(A, B\circ g)=\kappa(A, Bg)=\kappa(Ag^\top, B)=\kappa(\pi(A g^\top), B)$.
\end{Corollary}
Let $A,B\in V$, $g\in G$,
and let $A.g$ denote $\pi(A g^{-\top})$.
Then $V\to V:A\mapsto A.g$ is an $\mathbb{F}_q$-endomorphism,
and $\kappa|_{V\times V}(A.g,B)=\kappa|_{V\times V}(A,B\circ g^{-1})$.
By \cite[\S 2.1]{Markus1}, we get the following new action.
\begin{Corollary}\label{action A dot g-2G2}
There exists a unique linear action $-.-$ of $G$ on $V$:
\begin{align*}
-.- \colon V\times G  \to  V: (A,g)\mapsto A.g=\pi(A g^{-\top})
\end{align*}
such that
$\kappa|_{V\times V}(A.g,B)=\kappa|_{V\times V}(A,B\circ g^{-1})$
 for all  $B\in V$.
\end{Corollary}
\begin{Notation}  \label{f-2G2}
Set $f:=\pi|_G \colon G \to  V$.
\end{Notation}


\begin{Proposition}\label{f(xg)-2G2}
 Let $x, g \in G$. Then
 $f(x)g\equiv(x-1)g \mod V^\bot$
 and
$f(xg)=f(x)\circ g+f(g)$.
\end{Proposition}

\begin{proof}
Let $x,g\in G$.
Then $f(x)=\pi(x)
 \stackrel{\pi(1)=0}{=}\pi(x)-\pi(1)
\stackrel{\pi\text{ linear}}{=}\pi(x-1)$,
so $f(x)\equiv x-1 \mod V^\bot$.
For all $ A \in V$,
$
\kappa(A, (x-1)g)
{=}\kappa(Ag^\top, x-1)
{=}\kappa(\pi(Ag^\top), x-1)
{=}\kappa(\pi(Ag^\top), f(x))
=\kappa(Ag^\top, f(x))
=\kappa(A, f(x)g)
$.
Since $\kappa|_{V\times V}$ is a non-degenerate bilinear form,
we get that $f(x)g\equiv(x-1)g \mod V^\bot$.

We have that
$ f(xg)
{\equiv} xg-1
= (x-1)g+(g-1)
{\equiv} f(x)g+f(g) \mod V^\bot$,
so $\pi(f(xg))= \pi(f(x)g+f(g))$.
Thus, $f(xg)= \pi(f(x)g)+\pi(f(g))=f(x)\circ g + f(g)$.
\end{proof}

\begin{Proposition}[Bijective 1-cocycle of ${{^2}G}^{syl}_2(q)$]\label{f_U bij-2G2}
\index{1-cocycle!-of ${{^2}G}^{syl}_2(3^{2m+1})$}
If $U={{^2}G}^{syl}_2(3^{2m+1})$,
then $f|_U:=\pi|_U \colon  U \to  V $ is a bijection.
In particular, $f|_U$ is a bijective 1-cocycle of $U$.
\end{Proposition}
\begin{proof}
By \ref{sylow p subg-2G2}, $f|_U$ is bijective.
By \ref{f(xg)-2G2}, $f|_U$ is a 1-cocycle of $U$.
\end{proof}

\begin{Corollary}[Monomial linearisation for $A_8(q)$]
\label{mono. lin. G-2G2}
The map $f:=\pi|_G  \colon G \to  V$ is a surjective 1-cocycle of $G$ in $V$,
 and  $(f,{\kappa}|_{V\times V})$ is a monomial linearisation for $G=A_8(q)$.
\end{Corollary}

\begin{Corollary}
\label{mono. lin. for U-2G2}
 $(f|_U,{\kappa}|_{V\times V})$
is a monomial linearisation
for ${{^2}G}^{syl}_2(q)$.
\end{Corollary}

\begin{Notation}\label{vartheta}
 Let $\vartheta \colon \mathbb{F}_q^+ \to  \mathbb{C}^*$ denote a fixed nontrivial
 linear character of the additive group $\mathbb{F}_q^+$
 of $\mathbb{F}_q$
 once and for all.
In particular, $\sum_{x\in \mathbb{F}_q^+}{\vartheta(x)}=0$.
\end{Notation}

\begin{Lemma}\label{hat V=hat chi_A}
Let $A, B\in V$ and $\hat\chi_A \colon V \to  \mathbb{C}^*:X \mapsto  \vartheta\big(\kappa(A,X)\big)$.
Then
$
 \hat{V}
 =\{\hat\chi_A  \mid  A\in V \}
 =\mathrm{Irr}(V)
$,
and $\hat\chi_A=\hat\chi_B \iff A=B$.
In particular,
 $(\hat\chi_A).g=\hat\chi_{(A.g)}$
  for all $A\in V$ and $g\in G$.
\end{Lemma}

For $A\in V$, let $\chi_A:=\hat\chi_A f$.
Note that $\chi_{-A}(g)$ is equal to
the complex conjugate $\overline{\chi_A(g)}$ of ${\chi_A(g)}$
for all $A\in V$ and $g\in G$.
Now we establish the monomial $G$-module
$\mathbb{C}\left({{^2}G}^{syl}_2(3^{2m+1})\right)$.
\begin{Theorem}[Fundamental theorem for ${{^2}G}^{syl}_2(3^{2m+1})$]\label{fund thm U-2G2}
Let $G=A_8(q)$, $U={{^2}G}^{syl}_2(3^{2m+1})$
and
 \begin{equation*}
  [A]:=\frac{1}{|U|}\sum_{u\in U}{\overline{\chi_A(u)}u}
  =[\chi_{-A}|_U]
 \qquad \text{for all $A\in V$},
 \end{equation*}
where $\chi_{A}(u)=\vartheta\kappa(A,f(u))=\overline{\chi_{-A}(u)}$.
Then the set $\{[A] \mid A\in V\}$
forms a $\mathbb{C}$-basis for the complex group algebra $\mathbb{C}U$.
For all $g\in G,{\,} A\in V$,
let $[A]*g:=\chi_{A.g}(g)[A.g]=\vartheta\kappa(A.g,f(g))[A.g]$.
Then $\mathbb{C}U$ is a monomial $\mathbb{C}G$-module.
The restriction of the $*$-operation to $U$ is given by the usual right multiplication
  of $U$ on $\mathbb{C}U$, i.e.
\begin{align*}
 [A]*u=[A]u=\frac{1}{|U|}\sum_{y\in U}{\overline{\chi_A(y)}yu},
 \qquad \text{for all $u\in U,{\,} A\in V$}.
\end{align*}
\end{Theorem}
\begin{proof}
By \ref{mono. lin. G-2G2}, $(f,{\kappa}|_{V\times V})$ is a monomial linearisation for $G$
 satisfying that $f|_U$ is bijective (see \ref{f_U bij-2G2}).
By \ref{action A dot g-2G2}, $A.u=\pi(Au^{-\top})$.
Thus the theorem is obtained by \cite[2.1.35]{Markus1}.
\end{proof}

\begin{Comparison}[Monomial linearisations]
\label{com:monomial modules-3D4}
Let $U$ be
$A_n(q)$, $D_n^{syl}(q)$,
${^3}D_4^{syl}(q^3)$,
$G_2^{syl}(q)$
($q$  is a fixed power of some odd prime)
or $^2G_2^{syl}(q)$ $(q=3^{2m+1})$,
$G$ an intermediate group of $U$,
$V_0:=V_{\Squarb}$,
$V$ a subspace of $V_0$,
$J:=\mathrm{supp}(V)$,
$f\colon G\to V$ a surjective 1-cocycle of $G$
such that $f|_U$ is injective,
$\kappa\colon V \times V\to \mathbb{F}_q \text{ (or $\mathbb{F}_{q^3}$)}$
a trace form
such that
$(f,\kappa|_{V\times V})$ is a monomial linearisation for $G$.
Then the
monomial linearisations $(f|_U,{\kappa}|_{V\times V})$ for
$A_n(q)$ (see \cite[\S 2.2]{Markus1}),
$D_n^{syl}(q)$ (see \cite[\S 3.1]{Markus1}),
${^3}D_4^{syl}(q^3)$ (see \cite[\S 4]{sun3D4super}),
$G_2^{syl}(q)$ (see \cite[\S 3]{sunG2}),
and $^2G_2^{syl}(3^{2m+1})$  (see \S \ref{sec: monomial 2G2-module})
are listed in Table \ref{table:monomial linearisations}.
\begin{table}[!htp]
\caption{Monomial linearisations $(f|_U,{\kappa}|_{V\times V})$}
\label{table:monomial linearisations}
\begin{align*}
\begin{array}{|l|l|l|l|l|l|l|}
\hline
\multicolumn{1}{|c|}{U}
& \multicolumn{1}{c|}{G}
& \multicolumn{1}{c|}{V_0}
& \multicolumn{1}{c|}{J}
& \multicolumn{1}{c|}{V}
& \multicolumn{1}{c|}{f\colon G\to V}
& \multicolumn{1}{c|}{\kappa|_{V\times V}} \\\hline
A_n(q)
& A_{n}(q)
& \mathrm{Mat}_{n\times n}(q)
& \UR
& V=V_{\URb}
& f(g)=\pi_{\URb}(g)=g-I_n
& \kappa|_{V\times V}
\\\hline
D_n^{syl}(q)
& A_{2n}(q)
& \mathrm{Mat}_{2n\times 2n}(q)
& {\UP}
& V=V_{\UPb}
& f(g)=\pi_{\UPb}(g)
& \kappa|_{V\times V}
\\\hline
{^3}D_4^{syl}(q^3)
& G_8(q^3)
& \mathrm{Mat}_{8\times 8}(q^3)
& \tilde{J}
& V\neq V_{\tilde{J}}
& f(g)\neq \pi_{\tilde{J}}(g)
& \kappa_q|_{V\times V}
\\\hline
G_2^{syl}(q)
& G_8(q)
& \mathrm{Mat}_{8\times 8}(q)
& {\tilde{J}}
& V\neq V_{\tilde{J}}
& f(g)\neq \pi_{\tilde{J}}(g)
& \kappa|_{V\times V}
\\\hline
^2G_2^{syl}(3^{2m+1})
& A_8(3^{2m+1})
& \mathrm{Mat}_{8\times 8}(3^{2m+1})
& J
& V= V_J
& f(g)=\pi(g)= \pi_J(g)
& \kappa|_{V\times V}
\\\hline
\end{array}
\end{align*}
\end{table}
\end{Comparison}

From now on, we mainly consider the
regular right module
$(\mathbb{C}U,*)_{\mathbb{C}U}=\mathbb{C}U_{\mathbb{C}U}$.


\section{${^2}G_2^{syl}(q)$-orbit modules}
\label{sec:U-orbit modules-2G2}
Let
$p:=3$, $q:=3^{2m+1} {\ }(m\in \mathbb{N})$,
$U:={^2}G_2^{syl}(q)$, $A\in V$,
and
$y_i(t_i)\in U$, $t_i\in \mathbb{F}_{q}$
($i=t_1,t_3,t_4$).
In this section,
we determine the stabilizers $\mathrm{Stab}_U(A)$ for all $A\in V$ (\ref{prop: 2G2-stab})
and obtain a classification of $U$-orbit modules (\ref{prop:class orbit-2G2}).

For $A\in V$, the $U$\textbf{-orbit module}
 \index{orbit module}
 associated to $A$ is
$\mathbb{C}\mathcal{O}_U([A])
:=\mathbb{C}\{[A]u \mid u\in U\}
=\mathbb{C}\{[A.u] \mid u\in U\}$,
and the orbit modules are in fact modules due to the monomial action in Theorem \ref{fund thm U-2G2}.
Then $\mathbb{C}\mathcal{O}_U([A])$ has a $\mathbb{C}$-basis
$
\{[A.u] \mid  u\in U\}
          = \left\{[C] \mid C\in \mathcal{O}_U(A)\right\}$,
where
$\mathcal{O}_U(A):=\left\{A.g \mid  g\in U\right\}$
is the \textbf{orbit} of $A$ under the operation $-.-$ defined in \ref{action A dot g-2G2}.
The \textbf{stabilizer} \index{stabilizer!-of $A$}
 $\mathrm{Stab}_U(A)$ of $A$ in $U$ is
$\mathrm{Stab}_U(A)=\{u\in U  \mid  A.u=A\}$.
 Two $\mathbb{C}U$-modules having no nontrivial
 $\mathbb{C}U$-homomorphism between them are called
 \textbf{orthogonal}.
 Set $\tilde{x}_{ij}(t)= I_8+t e_{ij}\in A_8(q)$ $(1\leq i, j \leq 8)$.

\begin{Lemma}\label{2G2-A.xi, figures}
Let $A\in V$, $Y(t_1,t_3,t_4)\in U$ and $t_i \in \mathbb{F}_{q}$ with $i\in\{1,3,4\}$.
Then $A.Y(t_1,t_3,t_4)$
and
the corresponding figure of moves
are determined as follows.
 \begin{align*}
  A.Y(t_1,t_3,t_4)
= A.\big(y_2(t_1^{3\theta})y_1(t_1)y_3(t_3)\big)
= A.\big(\tilde{x}_{23}(t_1^{3\theta})\tilde{x}_{34}(t_1)\tilde{x}_{24}(t_3)\big).
 \end{align*}
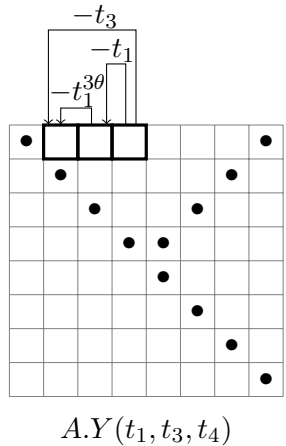
\begin{align*}
\begin{tikzpicture}[scale=0.90]
\draw[step=0.5cm, gray, very thin](-2, -2)grid(2, 2);
\draw (-1.75,1.75) node{$\bullet$};
\draw (-1.25,1.25) node{$\bullet$};
\draw (-0.75,0.75) node{$\bullet$};
\draw (-0.25,0.25) node{$\bullet$};
\draw (0.25,-0.25) node{$\bullet$};
\draw (0.75,-0.75) node{$\bullet$};
\draw (1.25,-1.25) node{$\bullet$};
\draw (1.75,-1.75) node{$\bullet$};
\draw (1.75,1.75) node{$\bullet$};
\draw (1.25,1.25) node{$\bullet$};
\draw (0.75,0.75) node{$\bullet$};
\draw (0.25,0.25) node{$\bullet$};
\foreach \x in {-1.5, -1, -0.5}
{
\draw [very thick] (\x,1.5)   rectangle +(0.5,0.5);
}
\draw[->] (-0.80,2)--(-0.80, 2.25)--(-1.25,2.25)--(-1.25,2.01);
\draw (-1.00,2.5) node{$-t_1^{3\theta}$};
\draw[->] (-0.30,2.0)--(-0.30, 2.90)--(-0.58,2.90)--(-0.58,2.01);
\draw (-0.50,3.10) node{$-t_1$};
\draw[->] (-0.15,2)--(-0.15, 3.40)--(-1.43,3.40)--(-1.43,2.01);
\draw (-0.75,3.60) node{$-t_3$};
\draw(0,-2.5) node{$A.Y(t_1,t_3,t_4)$};
\end{tikzpicture}
\end{align*}
The figure describes the way of classifying the orbits.
\end{Lemma}

The elements of $V$  are called \textbf{patterns}.
Let $A\in V$. Then
 $(i,j)\in J$ is a \textbf{main condition} \index{main condition} of $A$
if and only if $A_{ij}$ is the rightmost non-zero entry in the $i$-th row.
We set
$\mathrm{main}(A):=
 \{ (i,j)\in J  \mid  (i,j) \text{ is a main condition of }A\}$.
The \textbf{verge} of $A\in V$ is
$\mathrm{verge}(A):=
  \sum_{(i,j)\in \mathrm{main}(A)}{A_{i,j}e_{i,j}}$.
The  pattern $A\in V$ is called the \textbf{verge pattern}
if $A=\mathrm{verge}(A)$.

\begin{Notation}\label{family of orbit modules-2G2}
Define the families of $U$-orbit modules related to the main condition in the 1st row as follows:
$\mathfrak{F}_4:= \{\mathbb{C}\mathcal{O}_U([A]) \mid A\in V,{\,} A_{14}\neq 0\}$,
$\mathfrak{F}_3:= \{\mathbb{C}\mathcal{O}_U([A]) \mid A\in V,{\,} A_{13}\neq 0,{\,} A_{14}= 0\}$,
and
1-dimensional orbit modules
$\mathfrak{F}_{1}:= \{\mathbb{C}\mathcal{O}_U([A]) \mid A\in V,{\,} A_{12}\neq 0,{\,} A_{13}= A_{14}= 0\}$.
For $A\in V$, we also say $A\in \mathfrak{F}_i$
if $\mathbb{C}\mathcal{O}_U([A])\in \mathfrak{F}_i$.
\end{Notation}

\begin{Proposition}[$^{2}G_2^{syl}(q)$-orbit modules]\label{prop: 2G2-orbit}
For $A=(A_{ij})\in V$,
the $U$-orbit module $\mathbb{C}\mathcal{O}_U([A])$ is determined as follows.
\begin{align*}
&\mathbb{C}\mathcal{O}_U(
[A_{12}e_{12}+A_{13}e_{13}+A_{14}e_{14}])\\
=& \mathbb{C}
\{ [(A_{12}
-A_{13}t_1^{3\theta}
-A_{14}t_3)e_{12}
+
(A_{13}
-A_{14}t_1)e_{13}
+
A_{14}e_{14}
]
\mid
t_1, t_3\in \mathbb{F}_{q}\}.
\end{align*}
In particular,
every $U$-orbit module
contains
precisely one verge pattern.
\end{Proposition}
\begin{proof}
 By \ref{2G2-A.xi, figures},
 we calculate the orbit modules directly.
\end{proof}

\begin{Proposition}
\label{prop: 2G2-stab}
Let $A=(A_{ij})\in V$.
\begin{itemize}
\setlength\itemsep{0em}
 \item [(1)] If $A\in \mathfrak{F}_{1}$,
             then $\mathrm{Stab}_U(A)=U={^2}G_2^{syl}(q)$.
 \item [(2)] If  $A\in \mathfrak{F}_{3}$ and $A_{13}=A_{13}^*\in \mathbb{F}_{q}^*$,
             then $ \mathrm{Stab}_U(A)=\{Y(0,t_3,t_4)\mid t_3, t_4 \in \mathbb{F}_q\}$.
 \item [(3)] If  $A\in \mathfrak{F}_{4}$ and $A_{14}=A_{14}^*\in \mathbb{F}_{q}^*$,
 then
 $ \mathrm{Stab}_U(A)=\{Y(0,0,t_4)\mid t_4 \in \mathbb{F}_q\}$.
\end{itemize}
\end{Proposition}
\begin{proof}
 By \ref{prop: 2G2-orbit}, the stabilizers are obtained by straightforward calculation.
\end{proof}

Let $A, B \in V$,
$\mathrm{Stab}_U(A, B):=\mathrm{Stab}_U(A)\cap \mathrm{Stab}_U(B)$,
$\psi_A$ be the character of $\mathbb{C}\mathcal{O}_U([A])$
and $\psi_B$ denote the character of $\mathbb{C}\mathcal{O}_U([B])$.
Then
$\mathrm{Hom}_{\mathbb{C}U}(\mathbb{C}\mathcal{O}_U([A]),\mathbb{C}\mathcal{O}_U([B]))=\{0\}$
if and only if
$\mathrm{Hom}_{\mathrm{Stab}_U(C, D)}(\mathbb{C}[C],\mathbb{C}[D])=\{0\}$
for all $C\in \mathcal{O}_U(A)$ and $D\in \mathcal{O}_U(B)$.
 If $A, B\in V$,
then
$
\langle \psi_A, \psi_B\rangle_{U}
=\sum_{D\in \mathcal{O}_U(B)}
    \frac{|\mathrm{Stab}_U(A, D)|}
    {|\mathrm{Stab}_U(A)|}{\langle \chi_A, \chi_D\rangle_{\mathrm{Stab}_U(A, D)}}$,
where
$\chi_A$ and $\chi_D$ are the characters of the $\mathbb{C}\mathrm{Stab}_U(A, D)$-modules
$\mathbb{C}[A]$ and $\mathbb{C}[D]$ respectively.

\begin{Proposition}[Classification of ${^2}G_2^{syl}(q)$-orbit modules]\label{prop:class orbit-2G2}
Every $U$-orbit module is a verge pattern module in Table \ref{table:class orbit-2G2},
and the orbit modules satisfy the following properties.
\begin{itemize}
\setlength\itemsep{0em}
\item [(1)]
Let $A,B\in V$,
 $\mathrm{verge}(A)\neq \mathrm{verge}(B)$.
Then
$\mathrm{Hom}_{\mathbb{C}U}(\mathbb{C}\mathcal{O}_U([A]),\mathbb{C}\mathcal{O}_U([B]))=\{0\}$.
In particular, if $\mathbb{C}\mathcal{O}_U([A])\in \mathfrak{F}_i$,
$\mathbb{C}\mathcal{O}_U([B])\in \mathfrak{F}_j$ and $i\neq j$,
then
$\mathrm{Hom}_{\mathbb{C}U}(\mathbb{C}\mathcal{O}_U([A]),\mathbb{C}\mathcal{O}_U([B]))=\{0\}$.
\item [(2)]
In the family $\mathfrak{F}_{1}$, the $U$-orbit modules
are irreducible
and pairwise orthogonal.
\item [(3)]
In the family $\mathfrak{F}_{3}$ and $\mathfrak{F}_{4}$,
the $U$-orbit modules
are reducible.
\end{itemize}
\begin{table}[!htp]
\caption{${^2}G_2^{syl}(q)$-orbit modules}
\label{table:class orbit-2G2}
\begin{align*}
\begin{array}{|c|l|c|c|}\hline
\text{Family}
& \multicolumn{1}{c|}{\mathbb{C}\mathcal{O}_U([A]){\ } (A\in V)}
& \mathrm{dim}_{\mathbb{C}}\mathbb{C}\mathcal{O}_U([A])
& \text{Irreducible}\\\hline
\mathfrak{F}_4
& \mathbb{C}\mathcal{O}_U([A_{14}^*e_{14}]),
{\ }A_{14}\in \mathbb{F}_q^*
& q^2
& \text{NO} \\\hline
\mathfrak{F}_3
& \mathbb{C}\mathcal{O}_U([A_{13}^*e_{13}]),
{\ }A_{13}\in \mathbb{F}_q^*
& q
& \text{NO} \\\hline
\mathfrak{F}_1
& \mathbb{C}\mathcal{O}_U([A_{12}e_{12}]),
{\ }A_{12}\in \mathbb{F}_q
& 1
& \text{YES} \\\hline
 \end{array}
\end{align*}
\end{table}
\end{Proposition}
\begin{proof}
 By \ref{prop: 2G2-orbit}, every $U$-orbit module is a verge pattern module in Table \ref{table:class orbit-2G2}.
 By calculating the corresponding inner products, (1) and the orthogonal properties are proved.
 The orbit modules in the family $\mathfrak{F}_1$ are 1-dimensional and irreducible, so (2) is obtained.

Let $A=A_{12}e_{12}+A_{13}^*e_{13}\in \mathfrak{F}_3$
and $C\in \mathcal{O}_U(A)$.
By \ref{prop: 2G2-stab},
$\mathrm{Stab}_U(A,C)=\mathrm{Stab}_U(A)
=\{Y(0,t_3,t_4) \mid t_3,t_4\in \mathbb{F}_q\}$.
The inner product is
$\langle \chi_A, \chi_C\rangle_{\mathrm{Stab}_U(A, C)}=1$.
Let $\psi_A$ denote the character of $\mathbb{C}\mathcal{O}_U([A])$.
Then
\begin{align*}
& \mathrm{dim}_{\mathbb{C}}
   \mathrm{Hom}_{\mathbb{C}U}(\mathbb{C}\mathcal{O}_U([A]),\mathbb{C}\mathcal{O}_U([A]))
= \langle \psi_A, \psi_A\rangle_{U}\\
{=}& \sum_{C\in \mathcal{O}_U(A)}
   \frac{|\mathrm{Stab}_U(A, C)|}{|\mathrm{Stab}_U(A)|}
   \mathrm{dim}_{\mathbb{C}}\mathrm{Hom}_{\mathrm{Stab}_U(A, C)}(\mathbb{C}[ A ],\mathbb{C}[ C ])
= q > 1.
\end{align*}
Thus,
$\mathbb{C}\mathcal{O}_U([A])$ is not irreducible.

We claim that the orbit module $\mathbb{C}\mathcal{O}_U([A])$ is reducible
if $A\in V$ is a pattern of the family $\mathfrak{F}_4$.
Suppose that it is irreducible.
Then
$\big(\dim_{C} \mathbb{C}\mathcal{O}_U([A])\big)^2=q^{4}< |U|=q^{3}$.
This is a contradiction.
Thus the orbit modules of the family $\mathfrak{F}_4$ are reducible.
\end{proof}

\begin{Corollary}
 If $A,B\in V$, then
 $\mathbb{C}\mathcal{O}_U([A])=\mathrm{Res}_U^G \mathbb{C}\mathcal{O}_G([A])$,
 and
 the two orbit modules $\mathbb{C}\mathcal{O}_U([A])$ and $\mathbb{C}\mathcal{O}_U([B])$ are either
 isomorphic or orthogonal.
\end{Corollary}

\begin{Comparison}[Classification of orbit modules]
\label{com:classification and stabilizer-G2}
Every  $^2G_2^{syl}(q)$-orbit module
has precisely one verge pattern (see \ref{prop:class orbit-2G2})
that is similar to that of $A_n(q)$ (see \cite[Theorem 3.2]{yan2}).
However, this is
neither true
for ${{^3D}_4^{syl}}(q^3)$ (see \cite[\S 6]{sun3D4super})
nor true for $G_2^{syl}(q)$ (see \cite[\S 5]{sunG2}).
\end{Comparison}


\section{Conjugacy classes of ${^2}G_2^{syl}(3^{2m+1})$}
\label{conjugacy classes-2G2}
Let $p:=3$, $q:=3^{2m+1} {\ }(m\in \mathbb{N})$,
$G:=A_8(q)$, $U:={^2}G_2^{syl}(3^{2m+1})$,
and $t^*, t_1^*,t_3^*,t_4^*\in \mathbb{F}_q^*$.
In this section,
we determine the conjugacy classes of ${^2}G_2^{syl}(3^{2m+1})$ (\ref{prop:conjugacy classes-2G2}),
and obtain one partition of ${^2}G_2^{syl}(3^{2m+1})$ (\ref{superclass:ci ti-2G2})
which is a set of the superclasses proved in Section \ref{sec: supercharacter theories-2G2}.

If $x,u\in U$, then the conjugate of $x$ by $u$ is
${^u}x:=uxu^{-1}$,
and the conjugacy class of $u$ is
${^U}x:=\left\{vxv^{-1} {\,}\middle|{\,}  v\in U \right\}$.

\begin{Lemma}\label{varsigma_t,2G2}
 Let $t\in \mathbb{F}_q^*$, $\theta:=3^m$,
 $\mathbb{F}_{q}^{+}$ be the additive group of $\mathbb{F}_{q}$ and
 \begin{align*}
 \varsigma_{t} \colon  \mathbb{F}_{q}^{+} \to  \mathbb{F}_q^+: s\mapsto ts^{3 \theta}-t^{3\theta}s.
 \end{align*}
Then $\varsigma_{t}$ is a homomorphism and $|\mathrm{im}{\,\varsigma_{t}}|=\theta^2=3^{2m}$.
\end{Lemma}
\begin{proof}
We know that $\varsigma_{t}$ is a homomorphism.
{\it We claim that $\ker\varsigma_{t}=\{0,t,-t\}$}.
Note that $\ker\varsigma_{t}\supseteq \{0,t,-t\}$.
If $s\in \ker\varsigma_{t}$, then
$ts^{3 \theta}-t^{3\theta}s=0$.
We know $0\in \ker\varsigma_{t} \cap \{0,t,-t\}$.
If $s\neq 0$, then
 $(t^{-1}s)^{3\theta-1}=1
   \implies
 |t^{-1}s|\big| (3\theta-1)$.
Since $(t^{-1}s)^{q-1}=(t^{-1}s)^{3\theta^2-1}=1$,
the order $|t^{-1}s|$ divides the greatest common divisor
$(3\theta-1, 3\theta^2-1)=(3^{m+1}-1,3^{2m+1}-1)$.
If $m=0$, then $(3^{m+1}-1,3^{2m+1}-1)=(2,2)=2$.
If $m>0$, then $(3^{m+1}-1,3^{2m+1}-1)=2$
by the Euclidean algorithm.
So,
$|t^{-1}s|\big|2
  \implies
(t^{-1}s)^2=1
  \implies
t^{-1}s=\pm1
  \implies
s=\pm t$.
Thus, $\ker\varsigma_{t}\subseteq \{0,t,-t\}$.
The claim is proved.
Therefore,
$|\mathrm{im}{\,\varsigma_{t}}|=\frac{|\mathbb{F}_q^+|}{|\ker\varsigma_{t}|}
 =\frac{3^{2m+1}}{3}=3^{2m}$.
\end{proof}

\begin{Notation}
 For $t^*\in \mathbb{F}_q^*$,
 denote by ${^{t^*}}T$
  a transversal
 of $\mathrm{im}{\,\varsigma_{t^*}}$ in $\mathbb{F}_{q}^+$.
 Thus $|{^{t^*}}T|=3$.
\end{Notation}

\begin{Proposition}[Conjugacy classes of ${^2}G_2^{syl}(3^{2m+1})$]\label{prop:conjugacy classes-2G2}
If $U={^2}G_2^{syl}(3^{2m+1})$, then the conjugacy classes of $U$ are
listed in Table \ref{table:conjugacy classes-2G2}.
\begin{table}[!htp]
\caption{Conjugacy classes of ${^2}G_2^{syl}(3^{2m+1})$}
\label{table:conjugacy classes-2G2}
\begin{align*}
\begin{array}{|l|l|c|}\hline
\multicolumn{1}{|c|}{\text{Representative } y\in U}
& \multicolumn{1}{c|}{\text{Conjugacy Classes } {^Uy}}
&
\rule{0pt}{11pt}
|{^Uy}|
\\\hline
I_8 & Y(0,0,0) & 1\\\hline
Y(0,0,t_4^*),
{\ }t_4^*\in \mathbb{F}_q^*
& Y(0,0,t_4^*) & 1\\\hline
Y(0,t_3^*,0),
{\ }t_3^*\in \mathbb{F}_q^*
& Y(0,t_3^*,s_4),
{\ }
s_4\in \mathbb{F}_q
& q\\
\hline
Y(t_1^*,{^{t_1^*}}t_3,0),
{\ }t_1^*\in \mathbb{F}_q^*,
{\,} {^{t_1^*}}t_3\in {^{t_1^*}}T
& Y(t_1^*,\tilde{s}_3,s_4),
{\ }
\tilde{s}_3\in {^{t_1^*}}t_3+\mathrm{im}{\,\varsigma_{t_1^*}},
{\, }s_4\in \mathbb{F}_q
& q\cdot 3^{2m}\\
\hline
\end{array}
\end{align*}
\end{table}
\end{Proposition}

\begin{proof}
By \ref{prop:conjugacy classes of a,b,c-2G2}
and \ref{varsigma_t,2G2},
we get the conjugacy classes of $U$.
\end{proof}


\begin{Remark}
We consider the analogue of Higman's conjecture
for ${^2}G_2^{syl}(3^{2m+1})$.
By \ref{prop:conjugacy classes-2G2}, we get
$ \# \{\text{Conjugacy Classes of } {^2}G_2^{syl}(3^{2m+1})\}
 = 5q-4
 = 5(q-1)+1$.
Thus the conjecture is true for ${^2}G_2^{syl}(3^{2m+1})$.
\end{Remark}

\begin{Notation/Lemma}\label{superclass:ci ti-2G2}
Set
 \begin{align*}
  C_4(t_4^*):= {^U}Y(0,0,t_4^*),\quad
  C_3(t_3^*):= {^U}Y(0,t_3^*,0),\quad
  C_1(t_1^*):= \bigcup_{{^{t_1^*}}t_3\in {^{t_1^*}}T}
             {{^U}Y(t_1^*,{^{t_1^*}}t_3,0)}, \quad
  C_0: = \{ I_8 \}.
 \end{align*}
Note that the sets form a partition of $U$, denoted by $\mathcal{K}$, i.e.
\begin{align*}
\mathcal{K}:=\{C_0, C_1(t_1^*),C_3(t_3^*),C_4(t_4^*)  \mid  t_1^*,t_3^*,t_4^*\in\mathbb{F}_q^*\}.
\end{align*}
\end{Notation/Lemma}

\begin{Comparison}
\label{com:conjugacy classes-2G2}
\begin{itemize}
\setlength\itemsep{0em}
\item [(1)] (Superclasses).
The superclasses of ${^2G}_2^{syl}(q)$ in \ref{superclass:ci ti-2G2}
can also be obtained
 by calculating
 $\{I_8+x(u-I_8)y \mid x, y\in G\}\cap U
 =\{I_{8}+(u-I_{8})y\mid y \in A_8(q)\}\cap {^2G}_2^{syl}(q)$
 for all $u\in {^2G}_2^{syl}(q)$
 (c.f. \cite[\S 7]{sun3D4super} and \cite[\S 6]{sunG2}).
\item [(2)] (Conjugacy classes).
The conjugacy classes of $^2G_2^{syl}(q)$
are determined by commutator relations
that is similar to that of ${{^3D}_4^{syl}}(q^3)$ (see \cite[\S 3]{sun1})
and that of $G_2^{syl}(q)$ (see \cite[\S 8]{sunG2}).
\end{itemize}
\end{Comparison}


\section{A supercharacter theory for ${^2}G_2^{syl}(3^{2m+1})$}
\label{sec: supercharacter theories-2G2}
In this section,
let $p:=3$, $q:=3^{2m+1} {\ }(m\in \mathbb{N})$,
$U:={^2}G_2^{syl}(3^{2m+1})$,
and $A_{12}^*,A_{13}^*,A_{14}^*\in \mathbb{F}_q^*$.
We determine the supercharacter theory for ${^2}G_2^{syl}(3^{2m+1})$ (\ref{supercharacter theory-2G2}),
and establish the supercharacter table of ${^2}G_2^{syl}(3^{2m+1})$ in
Table \ref{table:supercharacter table-2G2}.

\begin{Definition}[{\cite[\S 2]{di}}/{\cite[3.6.2]{Markus1}}]
\label{supercharacter theory}
Let $G$ be a finite group.
Suppose that $\mathcal{K}$ is a partition of $G$
and that $\mathcal{X}$ is a set of (nonzero) complex characters of $G$,
such that
\begin{itemize}
 \setlength\itemsep{0em}
 \item [(a)] $|\mathcal{X}|=|\mathcal{K}|$,
 \item [(b)] every character $\chi \in \mathcal{X}$ is constant on each member of $\mathcal{K}$,
 \item [(c)] the elements of $\mathcal{X}$ are pairwise orthogonal and
  \item[(d)] the set $\{1\}$ is a member of $\mathcal{K}$.
\end{itemize}
Then $(\mathcal{X},\mathcal{K})$
is called a \textbf{supercharacter theory} for $G$.
We refer to the elements of $\mathcal{X}$ as \textbf{supercharacters},
and to the elements of $\mathcal{K}$ as \textbf{superclasses}
of $G$.
\end{Definition}

\begin{Notation}\label{notation:supermodules-2G2}
For $A=(A_{ij})\in V$,
set
$M{(0)}:= \{0\}$,
$M{(A_{12}^*e_{12})}:= \mathbb{C}\mathcal{O}_U([A_{12}^*e_{12}])$,
$M{(A_{13}^*e_{13})}:= \mathbb{C}\mathcal{O}_U([A_{13}^*e_{13}])$,
and
$M{({A_{14}^*}e_{14})}:= \mathbb{C}\mathcal{O}_U([A_{14}^*e_{14}])$.
Denote by $\mathcal{M}$
the set of all of the above $\mathbb{C}U$-modules, i.e.
\begin{align*}
\mathcal{M}:=&\{M(0),  M(A_{12}^*e_{12}),
         M(A_{13}^*e_{13}),
         M(A_{14}^*e_{14})
 \mid
A_{12}^*,A_{13}^*,A_{14}^*\in \mathbb{F}_q^*
\}\\
=&
\{\{0\}, \mathbb{C}\mathcal{O}_U([A_{12}^*e_{12}]),
         \mathbb{C}\mathcal{O}_U([A_{13}^*e_{13}]),
         \mathbb{C}\mathcal{O}_U([A_{14}^*e_{14}])
 \mid
A_{12}^*,A_{13}^*,A_{14}^*\in \mathbb{F}_q^*
\}.
\end{align*}
Note that $\mathcal{M}$ is a set of $U$-orbit modules.
\end{Notation}

\begin{Notation}\label{set of supercharacters-2G2}
For $M\in \mathcal{M}$, the complex character
of the $\mathbb{C}U$-module $M$ is denoted by $\Psi_M$.
We set
$\mathcal{X}:=\{\Psi_M \mid M\in \mathcal{M}\}$.
 Let $A\in V$, and $\psi_A$ be  the character of $\mathbb{C}\mathcal{O}_U([A])$.
 Then
$\Psi_{M(0)}= \psi_0$,
 $\Psi_{M(A_{12}^*e_{12})}= \psi_{A_{12}^*e_{12}}$,
 $\Psi_{M(A_{13}^*e_{13})}= \psi_{A_{13}^*e_{13}}$,
and
 $\Psi_{M(A_{14}^*e_{14})}= \psi_{A_{14}^*e_{14}}$.
\end{Notation}

\begin{Proposition}[Supercharacter theory for ${^2}G_2^{syl}(3^{2m+1})$]\label{supercharacter theory-2G2}
Let
$\mathcal{X}=\{\Psi_M \mid M\in \mathcal{M}\}$ defined in \ref{set of supercharacters-2G2}
 and
 $\mathcal{K}=\{C_0, C_1(t_1^*),C_3(t_3^*),C_4(t_4^*)  \mid  t_1^*,t_3^*,t_4^*\in\mathbb{F}_q^*\}$
defined in \ref{superclass:ci ti-2G2}.
Then $(\mathcal{X},\mathcal{K})$ is a supercharacter theory for ${^2}G_2^{syl}(3^{2m+1})$.
The supercharacter table is shown in Table \ref{table:supercharacter table-2G2}.
\begin{table}[!htp]
\caption{Supercharacter table of ${^2}G_2^{syl}(3^{2m+1})$}
\label{table:supercharacter table-2G2}
\index{supercharacter table!-of ${^2}G_2^{syl}(3^{2m+1})$}%
\begin{align*}
\begin{array}{l|cccc}
×
& C_0
& C_1(t_1^*)
& C_3(t_3^*)
& C_4(t_4^*)
\\
\hline
\Psi_{M{(0)}}
& 1 & 1 & 1 & 1 \\
\Psi_{M{(A_{12}^*e_{12})}}
& 1
& \vartheta (A_{12}^*t_1^*)
& 1
& 1 \\
\Psi_{M{(A_{13}^*e_{13})}}
& q
& 0
& q\cdot \vartheta (-A_{13}^*t_3^*)
& q \\
\Psi_{M{(A_{14}^*e_{14})}}
& q^2
& 0
& 0
& q^2\cdot \vartheta (-A_{14}^*t_4^*)
\end{array}
\end{align*}
\end{table}
\end{Proposition}
\begin{proof}
 By \ref{superclass:ci ti-2G2}, $\mathcal{K}$ is a partition of $U$.
 We know that $\mathcal{X}$ is a set of nonzero complex characters of $U$.
\begin{itemize}
\setlength\itemsep{0em}
 \item [(a)] {\it Claim that $|\mathcal{X}|=|\mathcal{K}|$}.
 By \ref{superclass:ci ti-2G2}, \ref{notation:supermodules-2G2} and \ref{set of supercharacters-2G2},
we have
$|\{\Psi_{M{(A_{13}^*e_{13})}} \mid  A_{13}^*\in \mathbb{F}_q^*\}|
=|\{{M{(A_{13}^*e_{13})}} \mid  A_{13}^*\in \mathbb{F}_q^*\}|
=|\{C_3(t_3^*) \mid t_3^* \in \mathbb{F}_q^* \}|$.
Thus $|\mathcal{X}|=|\mathcal{K}|=3(q-1)+1$.
 \item [(b)] {\it Claim that the characters $\chi \in \mathcal{X}$ are constant
                  on the members of $\mathcal{K}$}.

Let $A \in \mathfrak{F}_3$
(i.e. $A_{14}=0$, $A_{13}=A_{13}^*\in \mathbb{F}_q^*$)
and $y\in U$.
Then
 \begin{align*}
 \Psi_{M(A_{13}^*e_{13})}(y)=
     \sum_{\substack{C\in \mathcal{O}_{U}(A_{13}^*e_{13})\\C.y=C}}{\chi_C(y)}
 =\sum_{\substack{C\in \mathcal{O}_{U}(A_{13}^*e_{13})\\ y\in \mathrm{Stab}_U(C)}}{\chi_C(y)}.
 \end{align*}
 If $y=Y(0,t_3,t_4)\in C_0\cup C_3(t_3^*)\cup C_4(t_4^*)\subseteq \mathcal{K}$,
 then we have $y\in \mathrm{Stab}_U(C)$ for all $C\in \mathcal{O}_{U}(A_{13}^*e_{13})$ by \ref{prop: 2G2-stab}.
 Thus
\begin{align*}
\Psi_{M({A_{13}^*}e_{13})}(y)=
               \sum_{C\in \mathcal{O}_{U}(A_{13}^*e_{13})}{\chi_C(y)}
            =  q\cdot
               {\vartheta(-{A_{13}^*}t_3)}.
\end{align*}
If $y\in C_1(t_1^*)\subseteq \mathcal{K}$,
then $y \notin \mathrm{Stab}_U(C)$ for all $C\in \mathcal{O}_{U}(A_{13}^*e_{13})$ by \ref{prop: 2G2-stab}.
So
$\Psi_{M({A_{13}^*}e_{13})}(y)=0$.
Similarly, we calculate the other values of the Table \ref{table:supercharacter table-2G2}.
Thus, the claim is proved.
 \item [(c)] The elements of $\mathcal{X}$ are pairwise orthogonal by \ref{prop:class orbit-2G2}.
 \item [(d)] The set $\{I_8\}$ is a member of $\mathcal{K}$.
\end{itemize}
By \ref{supercharacter theory},
 $(\mathcal{X},\mathcal{K})$ is a supercharacter theory for ${^2}G_2^{syl}(3^{2m+1})$.
\end{proof}

\begin{Comparison}[Supercharacters]
\label{com:supercharacter theory-2G2}
Every supercharacter of $G_2^{syl}(q)$
is afforded by one orbit module
(see \ref{notation:supermodules-2G2},
\ref{set of supercharacters-2G2}
and
\ref{supercharacter theory-2G2})
that is analogous to that of
$A_n(q)$ (see \cite{yan2}).
However,
this holds neither for ${{^3D}_4^{syl}}(q^3)$ (see \cite[\S 8]{sun3D4super})
nor for $G_2^{syl}(q)$ (see \cite[\S 7]{sunG2}).
\end{Comparison}


\section{Character table of ${^2}G_2^{syl}(3)$}
\label{sec: character table-2G2-3}
In this section,
we exhibit
the conjugacy classes of ${^2}G_2^{syl}(3)$(see Table \ref{table:conjugacy classes-2G2(3)}),
the irreducible characters of ${^2}G_2^{syl}(3)$ (see Proposition \ref{construction of irr. char. of 2G2})
and the character table of ${^2}G_2^{syl}(3)$
(see Table \ref{table:character table-2G2(3)}).
Let $q=p=3$ (i.e. $m=0$) and $U:= {^2}G_2^{syl}(3)$.
Then $\theta=3^m=1$.

\begin{Notation}
Set
$Y_{a}  := \{a(t_1) \mid  t_1\in \mathbb{F}_q\}$,
$Y_{b}  := \{b(t_3) \mid t_3\in \mathbb{F}_q\}$,
 and
$Y_{c}  := \{c(t_4) \mid t_4\in \mathbb{F}_q\}$.
\end{Notation}

We recall the properties
for ${^2}G_2^{syl}(3)$.
\begin{Lemma}[Multiplication]
If $i\in \{1,3,4\}$ and $t_i, s_i\in \mathbb{F}_3$, then
\begin{align*}
 Y(t_1,t_3,t_4)\cdot Y(s_1,s_3,s_4)
=& Y(t_1+s_1,{\ }t_3+s_3-t_1s_1,{\ }t_4+s_4+t_1s_1^2-t_1^2s_1-t_3s_1),\\
 Y(t_1,t_3,t_4)^{-1}
=& Y(-t_1,{\ }-t_3-t_1^2,{\ } -t_4+t_1-t_1t_3).
\end{align*}
In particular,
\begin{alignat*}{2}
 a(t_1)\cdot a(s_1)=& Y(t_1+s_1,-t_1s_1,t_1s_1^2-t_1^2s_1),
 & \qquad
 & a(t_1)^{-1}= Y(-t_1, -t_1^2, t_1),\\
 b(t_3)\cdot b(s_3)=& b(t_3+s_3),
 & \qquad
 & c(t_4)\cdot c(s_4)= c(t_4+s_4).
\end{alignat*}
\end{Lemma}
\begin{Corollary}
$Y_{b}$ and $Y_{c}$ are subgroups of ${^2{G}}_2^{syl}(3)$,
but $Y_{a}$ is not a subgroup of  ${^2{G}}_2^{syl}(3)$.
\end{Corollary}

\begin{Lemma}[Commutator relations of ${^2}G_2^{syl}(3)$]
Let $i\in \{1,3,4\}$ and $t_i, s_i\in \mathbb{F}_3$.
Then the commutators of $^2{G}_2^{syl}(3)$ are
\begin{align*}
  [Y(t_1,t_3,t_4), Y(s_1,s_3,s_4)]
=&  Y\big(0,{\ }0,{\ }(t_1^2s_1-t_1s_1^2)+(t_1s_3-t_3s_1)\big),\\
  [Y(t_1,t_3,t_4)^{-1}, Y(s_1,s_3,s_4)^{-1}]
=& Y\big(0,{\ }0,{\ }
       (t_1^2s_1-t_1s_1^2)+(t_1s_3-t_3s_1)\big).
\end{align*}
{In particular, }
\begin{alignat*}{2}
& [a(t_1), a(s_1)]= c(t_1^2s_1-t_1s_1^2),
& \qquad
& [a(t_1)^{-1}, a(s_1)^{-1}]=  c(t_1^2s_1-t_1s_1^2),\\
& [a(t_1), b(s_3)]= c(t_1s_3),
& \qquad
& [a(t_1)^{-1}, b(s_3)^{-1}]= c(t_1s_3).
\end{alignat*}
\end{Lemma}

\begin{Corollary}
$Y_{c}$ and $Y_{b}Y_{c}$ are normal subgroups of ${^2{G}}_2^{syl}(3)$.
$Z(U)=Y_c$
and ${Y_c}\backslash U$ is abelian.
\end{Corollary}

\begin{Proposition}[Conjugacy classes of ${^2}G_2^{syl}(3)$]
\label{prop:conjugacy classes-2G2(3)}
If $t_i, s_i\in \mathbb{F}_q$ with $i\in \{1,3,4\}$,
then the conjugate of $Y(t_1,t_3,t_4)$ is
 \begin{align*}
 Y(s_1,s_3,s_4)\cdot Y(t_1,t_3,t_4)\cdot Y(s_1,s_3,s_4)^{-1}
=& Y\big(t_1,{\ }t_3,{\ }
t_4+(t_1s_1^2-t_1^2s_1)+(t_3s_1-t_1s_3)\big).
\end{align*}
In particular,
\begin{align*}
& Y(s_1,s_3,s_4)\cdot a(t_1)\cdot Y(s_1,s_3,s_4)^{-1}
= Y\big(t_1,{\ }0,{\ }
t_1s_1^2-t_1^2s_1-t_1s_3\big),\\
& Y(s_1,s_3,s_4)\cdot b(t_3)\cdot Y(s_1,s_3,s_4)^{-1}= Y(0,{\ }t_3,{\ }t_3s_1),\\
& Y(s_1,s_3,s_4)\cdot c(t_4)\cdot Y(s_1,s_3,s_4)^{-1}= c(t_4).
\end{align*}
Then the conjugacy classes of ${^2}G_2^{syl}(3)$ are
listed in Table \ref{table:conjugacy classes-2G2(3)}.
\begin{table}[!htp]
\caption{Conjugacy classes of ${^2}G_2^{syl}(3)$}
\label{table:conjugacy classes-2G2(3)}
\begin{align*}
\begin{array}{|l|l|c|}\hline
\multicolumn{1}{|c|}{\text{Representative } y\in U}
& \multicolumn{1}{c|}{\text{Conjugacy Classes } ^Uy}
& \rule{0pt}{11pt}
|{^Uy}|
\\\hline
I_8 & Y(0,0,0) & 1\\\hline
Y(0,0,t_4^*),
{\ }t_4^*\in \mathbb{F}_3^*
& Y(0,0,t_4^*) & 1\\\hline
Y(0,t_3^*,0),
{\ }t_3^*\in \mathbb{F}_3^*
& Y(0,t_3^*,s_4),
{\ }
s_4\in \mathbb{F}_3
& 3\\
\hline
Y(t_1^*,t_3,0),
{\ }t_1^*\in \mathbb{F}_3^*,
{\,} t_3 \in \mathbb{F}_3
& Y(t_1^*,t_3,s_4),
{\ } s_4\in \mathbb{F}_3
& 3\\
\hline
\end{array}
\end{align*}
\end{table}
\end{Proposition}

Now we explain the constructions of the irreducible characters of ${^2}G_2^{syl}(3)$.
\begin{Proposition}\label{construction of irr. char. of 2G2}
Let $U:= {^2}G_2^{syl}(3)$,
$\chi\in \mathrm{Irr}(U)$,
and
$A_{ij}\in \mathbb{F}_q$,
$A_{ij}^*\in \mathbb{F}_q^*$
$(1\leq i, j \leq 8)$.
\begin{itemize}
\setlength\itemsep{0em}
  \item [(1)]
  Let
$\bar{U}:={Y_c}\backslash U=\bar{Y}_a\bar{Y}_b$,
$\bar{\chi}_{lin}^{A_{12},A_{13}}\in \mathrm{Irr}(\bar{U})$,
$\bar{\chi}_{lin}^{A_{12},A_{13}}(\bar{a}(t_1)\bar{b}(t_3))
         :=\vartheta(A_{12}t_1)\cdot \vartheta(-A_{13}t_3)$,
and  $\chi_{lin}^{A_{12},A_{13}}$
  be the lift of $\bar{\chi}_{lin}^{A_{12},A_{13}}$ to $U$.
Then
\begin{align*}
\mathfrak{F}_{lin}
    := \{\chi\in \mathrm{Irr}(U) \mid  Y_c\subseteq \ker{\chi}\}
    = \{\chi_{lin}^{A_{12},A_{13}}  \mid A_{12}, A_{13}\in \mathbb{F}_{q}\}.
\end{align*}
  \item [(2)]
  Let
  $H:=Y_bY_c$,
  $\lambda^{A_{14},{A}_{13}}\in \mathrm{Irr}(H)$,
  $\lambda^{A_{14},{A}_{13}}(b(t_3)c(t_4))
         :=\vartheta(-A_{14}t_4-A_{13}t_3)$,
  ${\chi}_{2,q}^{A_{14}}:=
    \mathrm{Ind}_H^U \lambda^{A_{14},0}$.
Then
 $
   \mathfrak{F}_{2}
    := \{\chi\in \mathrm{Irr}(U) \mid  Y_c \nsubseteq \ker{\chi}\}
    = \{\chi_{2,q}^{A_{14}^*}
     \mid A_{14}^*\in \mathbb{F}_{q}^*\}$.
 \end{itemize}
 Hence $\mathrm{Irr}(U)=\mathfrak{F}_{lin}\dot{\cup} \mathfrak{F}_{2}$.
 \end{Proposition}

Note that when $q=3$,
the supercharacter $\Psi_{M{(A_{12}e_{12})}}$ afforded by $\mathbb{C}\mathcal{O}_U([A_{12}e_{12}])\in \mathfrak{F}_{1}$
(see \ref{family of orbit modules-2G2}
and
\ref{set of supercharacters-2G2})
is exactly the irreducible character $\chi_{lin}^{A_{12},0}\in \mathfrak{F}_{lin}$.
\begin{proof}
Let $\chi \in \mathrm{Irr}(U)$.
\begin{itemize}
\setlength\itemsep{0em}
\item [(1)]  {\it Family $\mathfrak{F}_{lin}$, where $Y_c \subseteq \ker\chi$}.

Since the commutator subgroup is $Y_c$,
all linear characters of $U$
are precisely the lifts of the irreducible characters of
the abelian quotient group ${Y_c}\backslash U$ to $U$.
\item [(2)]
{\it Family $\mathfrak{F}_{2}$, where $Y_c \nsubseteq \ker\chi$}.

 Let
   $H:=Y_bY_c$,
   $\lambda^{A_{14},{A}_{13}}\in \mathrm{Irr}(H)$
   with
   $\lambda^{A_{14},{A}_{13}}(b(t_3)c(t_4))
          =\vartheta(-A_{14}t_4-A_{13}t_3)$.
We note that $Y_a $ is a transversal of $H$ in $U$,
and that $Z(U)=Y_c$.
For all $s_1\in \mathbb{F}_q$,
\begin{align*}
 & \left(\lambda^{A_{14}^*,A_{13}}\right)^{a(s_1)}
   (b(t_3)c(t_4))
 =\lambda^{A_{14}^*,A_{13}}
   \big(a(s_1)\cdot b(t_3)c(t_4) \cdot a(s_1)^{-1}\big)\\
 =& \lambda^{A_{14}^*,A_{13}}
   \big( b(t_3)c(t_4+s_1t_3)\big)
 =\vartheta \big(-A_{14}^*(t_4+s_1t_3)-A_{13}t_3 \big)\\
 =&\vartheta \big(-A_{14}^*t_4-(A_{13}+s_1A_{14}^*)t_3 \big).
\end{align*}
Thus the inertia group
$I_{U}(\lambda^{A_{14}^*,A_{13}})=H$.
By Clifford's Theorem,
$\mathrm{Ind}_H^U \lambda^{A_{14}^*,A_{13}} \in \mathrm{Irr}(U)$ and
\begin{align*}
 \mathrm{Res}^U_H
 \mathrm{Ind}_H^U \lambda^{A_{14}^*,A_{13}}
 =\sum_{s_1\in \mathbb{F}_q}{\left(\lambda^{A_{14}^*,A_{13}}\right)^{a(s_1)}}
 =\sum_{s_1\in \mathbb{F}_q}{\lambda^{A_{14}^*,(A_{13}+s_1A_{14}^*)}}
 =\sum_{B_{13}\in \mathbb{F}_q}{\lambda^{A_{14}^*,B_{13}}}.
\end{align*}
Let ${\chi}_{2,q}^{A_{14}^*}:=\mathrm{Ind}_H^U \lambda^{A_{14}^*,0} $.
By Clifford theory,
there are $q-1$ almost faithful irreducible characters of ${U}$, i.e.
$\mathfrak{F}_{2}
    = \{\chi\in \mathrm{Irr}(U) \mid  Y_c \nsubseteq \ker{\chi}\}
    = \{\chi_{2,q}^{A_{14}^*}
     \mid A_{14}^*\in \mathbb{F}_{q}^*\}$.
\end{itemize}
\end{proof}

\begin{Proposition}
\label{prop:character table-2G2(3)}
The character table of ${^2}G_2^{syl}(3)$
is shown in Table \ref{table:character table-2G2(3)}.
\begin{table}[!htp]
\caption{Character table of ${^2}G_2^{syl}(3)$}
\label{table:character table-2G2(3)}
\index{character table!-of ${^2}G_2^{syl}(3)$}%
\begin{align*}
 \renewcommand\arraystretch{1.5}
\begin{array}{l|cccc}
×
& I_8
& Y(t_1^*,t_3,0)
& Y(0,t_3^*,0)
& Y(0,0,t_4^*)
\\
\hline
\chi_{lin}^{0,0}
& 1 & 1 & 1 & 1 \\
\chi_{lin}^{A_{12}^*,0}
& 1
& \vartheta (A_{12}^*t_1^*)
& 1
& 1 \\
\chi_{lin}^{A_{12},A_{13}^*}
& 1
& \begin{array}{l}
\vartheta (A_{12}t_1^*)
\cdot \vartheta (-A_{13}^*t_3)
  \end{array}
& \vartheta (-A_{13}^*t_3^*)
& 1 \\
\chi_{2,q}^{A_{14}^*}
& 3
& 0
& 0
& 3\cdot \vartheta (-A_{14}^*t_4^*)
\end{array}
\end{align*}
\end{table}
\end{Proposition}
\begin{proof}
We use the notation of Proposition \ref{construction of irr. char. of 2G2}.
Let $u=Y(t_1,t_3,t_4)\in U$,
$ H=Y_bY_c$,
and
$\lambda^{A_{14},{A}_{13}}\in \mathrm{Irr}(H)$
  with
$\lambda^{A_{14},{A}_{13}}(b(t_3)c(t_4))
         =\vartheta(-A_{14}t_4-A_{13}t_3)$.
We have
\begin{align*}
\chi_{2,q}^{A_{14}^*}(u)
=& \mathrm{Ind}_{H}^{U}\lambda^{A_{14}^*,0}(u)
= \frac{1}{|H|}
    \sum_{\substack{g\in U\\
             g\cdot Y(t_1,t_3,t_4)\cdot g^{-1}\in H} }
    \lambda^{A_{14}^*,0}\big(g\cdot Y(t_1,t_3,t_4)\cdot g^{-1}\big).
\end{align*}
Then
\begin{align*}
\chi_{2,q}^{A_{14}^*}(c(t_4))=& q\cdot \lambda^{A_{14}^*,0}(c(t_4))
                             =q\cdot\vartheta(-A_{14}^*t_4)
                             =3\cdot\vartheta(-A_{14}^*t_4),\\
\chi_{2,q}^{A_{14}^*}(b(t_3^*))=&  \frac{1}{|H|}
                                   \sum_{\substack{g:=Y(s_1,s_3,s_4)\in U\\
                                    g\cdot b(t_3^*)\cdot g^{-1}\in H} }
                                    \lambda^{A_{14}^*,0}\big(g\cdot b(t_3^*)\cdot g^{-1}\big)
                               =  \sum_{s_1\in \mathbb{F}_q }
                                    \lambda^{A_{14}^*,0}\big(Y(0,t_3^*, t_3^*s_1)\big)\\
                               =&  \sum_{s_1\in \mathbb{F}_q }
                                   {\vartheta(-A_{1,4}^* t_3^*s_1)\big)}
                               = 0,\\
\chi_{2,q}^{A_{14}^*}(Y(t_1^*,t_3,0))=& 0.
\end{align*}
All the other values are determined similarly.
\end{proof}

\begin{Proposition}[Supercharacters and irreducible characters]
\label{relation: superchar. and irr.-2G2(3)}
If $q=3$, then the relations between supercharacters and irreducible characters of ${^2}G_2^{syl}(3)$
are established.
\begin{align*}
\Psi_{M{(A_{14}^*e_{14})}}= 3 \cdot \chi_{2,q}^{A_{14}^*},
 \quad
\Psi_{M{(A_{13}^*e_{13})}}= \sum_{A_{12}\in \mathbb{F}_q} \chi_{lin}^{A_{12},A_{13}^*},
\quad
\Psi_{M{(A_{12}^*e_{12})}}= \chi_{lin}^{A_{12}^*,0},
 \quad
\Psi_{M{(0)}}= \chi_{lin}^{0,0}.
\end{align*}

\end{Proposition}

\begin{proof}
Compare Table \ref{table:supercharacter table-2G2}
and Table \ref{table:character table-2G2(3)},
the formulae are obtained.
\end{proof}

\begin{Remark}
Let
$q=3$,
and
$\#\mathrm{Irr}_c$ be the number of irreducible characters of ${^2}G_2^{syl}(3)$
of dimension $q^c$ with $c\in \mathbb{N}$.
Then
$\#\mathrm{Irr}_1  = q-1=2$,
$\#\mathrm{Irr}_0  = q^2   = (q-1)^2+2(q-1)+1=9$,
and
\begin{align*}
  & \# \{\text{Irreducible Characters of } {^2}G_2^{syl}(3) \}
 = \# \{\text{Conjugacy Classes of } {^2}G_2^{syl}(3)\}\\
 =& q^2+q-1
 = (q-1)^2+3(q-1)+1
 =5(q-1)+1
 =11.
\end{align*}
We consider the analogue of Higman's conjecture,
Lehrer's conjecture and Isaacs' conjecture of $A_n(q)$
for ${^2}G_2^{syl}(3)$.
The conjectures hold for ${^2}G_2^{syl}(3)$.
\end{Remark}

\begin{Comparison}[Irreducible characters]
\label{com:character table-G2}
The irreducible characters of ${^2}G_2^{syl}(3)$ are determined by Clifford theory
that is similar to that of ${{^3D}_4^{syl}}(q^3)$ (see \cite{Le} and \cite[\S 4]{sun1})
and that of $G_2^{syl}(q)$ (see \cite[\S 9]{sunG2}).
\end{Comparison}


\section*{Acknowledgements}
This paper is a part of my PhD thesis \cite{sunphd} at the University of Stuttgart, Germany,
so I am deeply grateful to my supervisor Richard Dipper.
I also would like to thank  Jun Hu
and Mathias Werth
for the helpful discussions and valuable suggestions.


\bibliographystyle{alpha}
\bibliography{bibliography2G2upercharacterandConjclasses}

\end{document}